\documentclass[11pt]{article}
\usepackage{amsthm}
\usepackage{amsmath}
\usepackage{amssymb}
\usepackage[toc,page]{appendix}

\allowdisplaybreaks

\makeatletter
\@addtoreset{equation}{section}

\makeatother

\makeatletter

\newdimen\mainfontsize \mainfontsize=1\@ptsize pt
\topskip=\mainfontsize
   \@maxdepth=\maxdepth

\makeatother

\theoremstyle{plain}
\newtheorem{thm}{Theorem}[section]
\newtheorem{lem}[thm]{Lemma}

\newtheorem{cor}[thm]{Corollary}
\theoremstyle{definition}
\newtheorem{defn}[thm]{Definition}

\newtheorem{ass}[thm]{Assumption}
\theoremstyle{remark}
\newtheorem{rem}[thm]{Remark}

\DeclareMathOperator*{\esssup}{ess\,sup}
\allowdisplaybreaks
\numberwithin{thm}{section}
\numberwithin{equation}{section}

\title{Stochastic PDEs for large portfolios with general mean-reverting volatility processes}
\author{Ben Hambly\footnote{Mathematical Institute, University of Oxford, e-mail: hambly@maths.ox.ac.uk} $\,$ and Nikolaos Kolliopoulos\footnote{Department of Mathematical Sciences, Carnegie Mellon University, e-mail: nkolliop@andrew.cmu.edu (corresponding author)} \\
Mathematical Institute, University of Oxford \\
Department of Mathematics, University of Michigan}
\date{\today} 

\begin{document}
\maketitle

%
%
%
%
%
%

\begin{abstract}
We consider a structural stochastic volatility model for the loss from a large portfolio of credit risky assets. Both the asset value and the volatility processes are correlated through systemic Brownian motions, with default determined by the asset value reaching a lower boundary. We prove that if our volatility models are picked from a class of mean-reverting diffusions, the system converges as the portfolio becomes large and, when the vol-of-vol function satisfies certain regularity and boundedness conditions, the limit of the empirical measure process has a density given in terms of a solution to a stochastic initial-boundary value problem on a half-space. The problem is defined in a special weighted Sobolev space. Regularity results are established for solutions to this problem, and then we show that there exists a unique solution. In contrast to the CIR volatility setting covered by the existing literature, our results hold even when the systemic Brownian motions are taken to be correlated. 
\end{abstract}

\section{Introduction}

A key quantity in the modeling of large credit portfolios is the loss process, the proportion of the assets that have defaulted as a function of time. This is a critical component required for the pricing of credit indices \cite{BHHJR,Sch} and asset backed securities \cite{AHL}, as well as in models for systemic risk, \cite{NS,HS}.
A structural approach to credit portfolios is to model the values of the assets as particle positions satisfying a system of correlated stochastic differential equations (SDEs), with an asset defaulting when its value hits a lower boundary. The asset values are driven by their own idiosyncratic noises and a common systemic noise which models macroeconomic effects on the whole system. By considering the empirical measure one can capture the evolution of the whole portfolio and, in particular, that of the loss process. Taking the limit as the number of assets tends to infinity, the idiosyncratic noises are averaged out, and the asymptotic behaviour of the empirical measure is described by a stochastic partial differential equation (SPDE) driven by the systemic noise in a half-space. 
Other approaches to the modelling of credit risk in large portfolios which also lead to the asymptotic analysis of an interacting particle system include the use of reduced form models, where the position of each particle describes the default intensity of an asset \cite{GSS,GSSS,SSG}. In this paper we study a class of structural credit portfolio models with stochastic volatility, for which the values of the loss process in the large portfolio limit (when the number of assets tends to infinity) can be estimated by solving the corresponding limiting SPDE numerically.



A simple constant volatility structural model for a large credit portfolio was studied in \cite{BHHJR}, and extended 
in \cite{HL16} to the case where the correlation of the values of any two assets depends on the total loss. Models 
of this kind incorporating stochastic volatility were introduced in \cite{HK17}, where volatilities were modelled 
as CIR processes. In \cite{HK17, ERR} we were able to prove the existence of an integrable density for 
the empirical measure limit which satisfies the half-space SPDE. However Sobolev regularity and a suitable boundary condition could not be established except under the very strong assumption that the systemic noises driving the value processes and the volatilities were uncorrelated. Furthermore, even with some weighted 
Sobolev regularity and the condition of vanishing at the boundary, we were unable to prove uniqueness of solutions 
to the SPDE. In this paper we adapt that work to a different class of stochastic volatility models, which contains 
the Ornstein-Uhlenbeck process, and which leads to a well-posed stochastic 
initial-boundary value problem without the need to impose any unrealistic conditions. This improves the applicability 
of the model, while the analysis in the weighted Sobolev space that needs to be used is novel and interesting from a 
mathematical point of view.

We will work with a general form of a stochastic volatility structural large portfolio model, in which we have a collection of $N$ credit risky assets, where the value process $A^i$ of the $i$-th asset satisfies the system of SDEs
\begin{equation}
\begin{array}{rcl}
dA_{t}^{i} &=& A_{t}^{i}{\mu}_{i}dt+A_{t}^{i}h\left(\sigma_{t}^{i}\right)\left(\sqrt{1-\rho_{1,i}^{2}}dW_{t}^{i}+\rho_{1,i}dW_{t}^{0}\right),\;0\leq t\leq T_{i}\\
d\sigma_{t}^{i} &=& k_{i}(\theta_{i}-\sigma_{t}^{i})dt+\xi_{i}q\left(\sigma_{t}^{i}\right)\left(\sqrt{1-\rho_{2,i}^{2}}dB_{t}^{i}+\rho_{2,i}dB_{t}^{0}\right),\;t\geq 0\\
A_{t}^{i}&=&b^{i},\;t>T_{i}\\
(A_{0}^{i},\,\sigma_{0}^{i}) &=& (a^{i},\,\sigma^{i}),
\end{array}
\label{eq:modelA}
\end{equation}
for all $i\in\{1,2,\dots,N\}$, where $T_{i}:=\inf\{t\geq0:\,A_{t}^{i}=b^{i}\}$ models the time of the $i$-th asset's default, $a^{1},\,a^{2},\,\ldots,\,a^{N}$ and $\sigma^{1},\,\sigma^{2},\,\ldots,\,\sigma^{N}$
are the initial values of the asset values and the volatilities respectively,
$b^{i} \leq a^{i}$ is the default barrier for the value of the $i$-th asset, $h$ 
and $q$ are functions with enough regularity (specified in Assumption~\ref{ass1} and Theorem~\ref{thm:4.2}), and $W_{t}^{0},\,B_{t}^{0},\,\ldots,\,W_{t}^{N},\,B_{t}^{N}$
are standard Brownian motions. The processes $W_{t}^{0}$ and $B_{t}^{0}$ are used to capture the systemic risk, while $W_{t}^{i}$ and $B_{t}^{i}$ for $i \geq 1$ represent the random movements in the $i$-th asset's value due to idiosyncratic features. We denote by $C_{i}=(k_{i},\,\theta_{i},\,\xi_{i},
\,r_{i},\,\rho_{1,i},\,\rho_{2,i})$ the vector of the coefficients of the SDEs that drive the $i$-th asset's value and volatility processes, and we assume that $b^{i}$, $(a^{i},\,\sigma^{i})$ and $C_{i}$ are drawn independently from some appropriate distributions for each $i$, the Brownian motions are independent from each $a^{i},\,\sigma^{i}$ and $C_{i}$, and among the Brownian motions only $W_{t}^{0}$ and $B_{t}^{0}$ are allowed to have a non-zero correlation. Next, we consider the corresponding logarithmically scaled particle system, which is obtained by setting $X_{t}^{i}=\left(\ln A_{t}^{i}-\ln b^{i}\right)$ in \eqref{eq:modelA} and by using Ito's formula to derive the following equivalent system of SDEs
\begin{equation}
\begin{array}{rcl}
dX_{t}^{i} &=& \left(r_{i}-\frac{h^{2}(\sigma_{t}^{i})}{2}\right)dt+h(\sigma_{t}^{i})\left(\sqrt{1-\rho_{1,i}^{2}}dW_{t}^{i}+\rho_{1,i}dW_{t}^{0}\right),\;\;0\leq t\leq T_{i} \\
d\sigma_{t}^{i} &=& k_{i}(\theta_{i}-\sigma_{t}^{i})dt+\xi_i\sqrt{1-\rho_{2,i}^{2}}q\left(\sigma_{t}^{i}\right)dB_{t}^{i}+\xi_i\rho_{2,i}q\left(\sigma_{t}^{i}\right)dB_{t}^{0},\;\;t\geq 0 \\
X_{t}^{i} &=& 0,\;t>T_{i} \\
(X_{0}^{i},\,\sigma_{0}^{i}) &=& (x^{i},\sigma^{i}),
\end{array}
\label{eq:model}
\end{equation}
for $i\in\{1,\,2,\,\ldots,\,N\}$, where $x^{i}=\left(\ln a^{i}-\ln b^{i}\right) \geq 0$
and $T_{i}=\inf\{t\geq0:\,X_{t}^{i}=0\}$ $\forall\,1\leq i\leq N$.
The process $X^i$ is usually thought of as the distance to default of the $i$-th asset.

We will study the large portfolio limit, that is the limit as $N \to \infty$, of the empirical measure of the above interacting particle system. 
We consider the empirical measure process given by
\begin{equation}\label{eq:1.5}
v_{t}^{N}=\frac{1}{N} \sum_{i=1}^{N}\delta_{X_{t}^{i},\sigma_{t}^{i}},
\end{equation}
on $\mathbb{R}^2$ (which has no mass on $(-\infty, 0)\times\mathbb{R}$), and its restriction to $(0,\infty)\times\mathbb{R}$ which is given by 
\begin{equation}\label{eq:1.6}
v_{1,t}^{N}=\frac{1}{N} \sum_{i=1}^{N}\delta_{X_{t}^{i},\sigma_{t}^{i}}\mathbb{I}_{\{T_{i}>t\}},
\end{equation}
for $t \geq 0$. The total mass of $v_{t}^{N}$ is always equal to $1$, while the total mass of $v_{2,t}^{N} := v_{t}^{N} - v_{1,t}^{N}$ is the loss process of our portfolio, which takes values in $\left[0, \, 1\right]$. 
We will see that the convergence results established in \cite{HK17} for the CIR volatility case (in which $q(z) = \sqrt{z}$) hold in the general case as well. That is, for some $\sigma$-algebra $\mathcal{G}$ contained in the $\sigma$-algebra generated by the initial data, we have both
\begin{eqnarray*}
v_{t}^{N}\to v_{t} = \mathbb{P}\left(\left(X_{t}^{1}, \, \sigma_{t}^{1}\right)\in\cdot\,|\,W^{0},\,B^{0},\,\mathcal{G}\right)
\end{eqnarray*}
and
\begin{eqnarray*}
v_{1,t}^{N}\to v_{1,t} &=& \mathbb{P}\left(\left(X_{t}^{1}, \, \sigma_{t}^{1}\right)\in\cdot,T_{1}>t\,|\,W^{0},\,B^{0},\,\mathcal{G}\right)
\end{eqnarray*}
weakly as $N\to\infty$, for all $t \geq 0$, $\mathbb{P}$-almost surely. We can write
\begin{eqnarray*}
v_{1,t} = \mathbb{E}\left[v_{t, \, C_{1}}\left(\cdot\right)\,|\,W^{0},\,B^{0},\,\mathcal{G}\right]
\end{eqnarray*}
where we define $v_{t, \, C_{1}}\left( \cdot,\right) := \mathbb{P}\left(\left(X_{t}^{1}, \, \sigma_{t}^{1}\right)\in\cdot,T_{1}>t\,|\,W^{0},\,B^{0}, \, C_{1}, \, \mathcal{G} \right)$. Given a realization of the coefficient vector $C_{1}$, if the measure-valued process $v_{t, \, C_{1}}\left( \cdot,\right)$ has a density $u_{t,C_{1}}$, it will satisfy the SPDE
\begin{eqnarray}\label{spdeweak1}
u_{t,C_{1}}(x,\,y)&=&u_{0}(x,\,y)-\int_{0}^{t}\left(r_1-\frac12 h^{2}(y)\right)\left(u_{s,C_{1}}(x,\,y)\right)_{x}ds \nonumber \\
& & -\int_{0}^{t}k_{1}\left(\theta_{1}-y\right)\left(u_{s,C_{1}}(x,\,y)\right)_{y}ds \nonumber \\
& & +\frac{1}{2}\int_{0}^{t}h^{2}(y)\left(u_{s,C_{1}}(x,\,y)\right)_{xx}ds \nonumber \\
&  & +\frac{\xi_{1}^{2}}{2}\int_{0}^{t}\left(q^2\left(y\right)u_{s,C_{1}}(x,\,y)\right)_{yy}ds \nonumber \\
& & +\xi_{1}\rho_{3}\rho_{1,1}\rho_{2,1}\int_{0}^{t}\left(h(y)q\left(y\right)u_{s,C_{1}}(x,\,y)\right)_{xy}ds 
\nonumber \\
& &  -\rho_{1,1}\int_{0}^{t}h(y)\left(u_{s,C_{1}}(x,\,y)\right)_{x}dW_{s}^{0} \nonumber \\
& & -\xi_{1}\rho_{2,1}\int_{0}^{t}\left(q\left(y\right)u_{t,C_{1}}(x,\,y)\right)_{y}dB_{s}^{0}, \qquad \quad \label{eq:spde1}
\end{eqnarray}
in some weak sense in $\mathbb{R}^{+} \times \mathbb{R}$, where $u_0$ is the initial density and $\rho_{3} := \int_0^1dW_s^0dB_s^0$ is the correlation coefficient between $W_{t}^{0}$ and $B_{t}^{0}$. As in \cite{HK17}, the aim is to derive some boundary condition on the axis $x = 0$, so we can then solve the above half-space SPDE for a random sample $\{c_{1},\,c_{2},\,\ldots,\,c_{n}\}$ of the coefficient vector $C_{1}$, and then approximate the loss process from
\begin{eqnarray}
\lim_{N\rightarrow\infty}v_{t}^{N}(\{0\}\times\mathbb{R})&=&1-\lim_{N\rightarrow\infty}v_{1,t}^{N}(\mathbb{R}^{2}) \nonumber \\
&=& 1-\mathbb{E}\left[\int_{0}^{\infty}\int_{\mathbb{R}}u_{t,C_{1}}(x,\,y)dxdy,\,|\,W^{0},\,B^{0},\,\mathcal{G}\right] \nonumber
\\
&\thickapprox & 1-\frac{1}{n}\sum_{i=1}^{n}\int_{0}^{\infty}\int_{\mathbb{R}}u_{t,c_{i}}(x,\,y)dxdy. \label{eq:1.8}
\end{eqnarray}
This approximation could then be used for pricing CDO tranches, whose payoffs are piecewise linear functions of this loss process, as well as for the computation of particular risk measures in a large market modelled as a large portfolio.   

Of course, the parameters of the distribution from which each $C_i$ is picked have to be estimated and, when the $C_i$s can take different values, $n$ in \eqref{eq:1.8} has to be sufficiently large to give an accurate approximation. A natural way to calibrate our model is to simply assume that $C_i$ equals the same deterministic vector $C$ for all $i \in \mathbb{N}$, solve the initial-boundary value problem numerically to estimate the loss process from
\[
\lim_{N\rightarrow\infty}v_{t}^{N}(\{0\}\times\mathbb{R}) = 1-\int_{0}^{\infty}\int_{\mathbb{R}}u_{t,C}(x,\,y)dxdy 
\]
for many different values of $C$, and finally minimize the least squares distance between model and market prices of CDO tranches to determine the best fit value of the parameter vector $C$.
The empirical weak limit $v_{1,t}$ in this deterministic coefficients setting will coincide with the measure-valued process $v_{t,C}$ whose density $u_{t,C}$ satisfies our SPDE, just as in \cite{BHHJR}. Therefore, the multilevel Monte Carlo method used in \cite{GR} for the model studied in \cite{BHHJR}, if extended appropriately, could also be used for our model to speed up the approximation of CDO tranche prices. 

In order to recover the density $u_{t, \, C_{1}}$ from that SPDE, we need to determine its boundary behaviour. The density of the empirical measure limit derived in \cite{BHHJR} for a constant volatility structural model was shown to vanish at $x = 0$ and thus, we expect this to be the case in the stochastic volatility setting as well. This makes $u_{t, \, C_{1}}(0, y) = 0$ for all $t > 0$ and $y \in \mathbb{R}$ the natural candidate for the boundary condition of our stochastic initial-boundary value problem, in the absence of which our SPDE \eqref{spdeweak1} is only satisfied in a very weak sense, where the derivatives are defined as distributions over a space of test functions $\phi$ with $\phi(0) = \phi'(0) = 0$. This expected boundary behaviour was derived in \cite{HK17,ERR} for the CIR volatility case along with some weighted Sobolev regularity, but both results could only be established when $\rho_3 = 0$. Even in the classical Heston model used for simple option pricing, the noises driving the price of a single asset and its volatility are generally taken to be correlated, see \cite{FOU}. Furthermore, even when $\rho_3 = 0$ holds, the uniqueness of solutions to the stochastic initial-boundary value problem derived in \cite{HK17,ERR} is still an open question, a potential issue for the numerical implementation of the model.

In the CIR volatility case, the condition $\rho_3 = 0$ was used as both the regularity and the boundary behaviour of $u_{t, \, C_{1}}$ had to be transferred from that of the one-dimensional density obtained in \cite{BHHJR} in a constant volatility setting. This is achieved by conditioning on the volatility path and approximating it by discrete paths, but the inability to derive a deterministic bound for the density of the volatility given $B^0$ adds a substantial difficulty to that as long as conditioning on $B^0$ affects the distribution of $W^0$, that is, as long as the two systemic noises are correlated. The challenge in proving uniqueness of solutions to the half-space SPDE even when $\rho_3 = 0$, comes from the fact that the unbounded SPDE coefficients have different growth rates. In this paper we show that the above problems can be circumvented when the vol-of-vol (volatility of the volatility) function $q$ is sufficiently bounded with fast-decaying derivatives. 

The significance of this work to the existing literature on structural large credit portfolio models is to incorporate correlation and obtain uniqueness of solutions for a class of mean-reverting stochastic
volatility models which includes the Ornstein-Uhlenbeck volatility model (the case when $q$ is a constant function). 
Also, our results are applicable to a class of effective approximations to a CIR model, with $q$ being, for example, 
equal to the square root function in a large compact subinterval of $\left(0, \, +\infty\right)$. A different approach 
to dealing with the issues arising for each stochastic volatility model in this setting is a fast mean-reverting
volatility asymptotic analysis which is the subject of \cite{NK18}. 

The results we obtain are interesting from a mathematical point of view as well. The convergence results for the empirical measure and the existence of a regular density satisfying the stochastic initial-boundary problem (Sections 2,3,4 and 5) are adaptations of the CIR volatility case treated in \cite{HK17,ERR}. However, in order to cope with the correlated noises, we need to extend \cite[Theorem 4.1]{HK17} to the case when the volatility process $\{\sigma_t\}_{t \in \left[0, \, T\right]}$ depends on the systemic noise $W^0$ driving the values of the portfolio assets. Moreover, our stochastic initial-boundary value problem has the interesting property that in order to establish higher regularity and then uniqueness of solutions, we need to impose precisely the regularity and weighted integrability conditions which can be established for the density of the empirical measure limit. Finally and most importantly, we are able to establish the uniqueness of solutions to our problem. Indeed, existing works on SPDEs in half-spaces with boundary conditions (e.g \cite{Krylov}) do not consider equations with unbounded coefficients like the $k_1(\theta_1 - y)$ term in \eqref{spdeweak1}, while the nature of our problem forces us to work in an intersection of weighted Sobolev spaces with special features. In particular, because of a weight vanishing at $x = 0$, the $L^2$ norm we use for the first order derivative in $x$ is weaker than the standard $L^2$ norm, and an alternative to the use of the standard trace operator is needed to define the boundary condition. A key for establishing the uniqueness result is a change in the probability measure which eliminates a term which we would be unable to control otherwise.

The structure of the paper is as follows: In Section 2, we discuss the asymptotic behaviour of the empirical measure. Next, in Section 3, we state the result on the existence of a density for the volatility given the systemic noise $B^0$, which includes also the existence of the required bounds for that density. In Section 4 we explain the improvement of \cite[Theorem 4.1]{HK17}, and then we derive the full two-dimensional density satisfying the SPDE and the boundary condition. Section 5 is devoted to obtaining higher regularity for solutions to our stochastic initial-boundary value problem along with some related lemmas, and in Section 6 we establish the crucial uniqueness of solutions. Finally, we have an appendix where we prove the result of Sections 2 and 3 along with the lemmas of 
Section~5.

\section{Convergence of the empirical measure processes and the limiting SPDE}

In order to state our convergence results, we need to describe our set up in more detail first. Let $(\Omega,\,\mathcal{F},\,\{\mathcal{F}_{t}\}_{t\geq0},\,\mathbb{P})$
be a complete filtered probability space, which can be decomposed as a product of three independent probability spaces representing the three different sources of randomness as follows
\begin{equation*}
(\Omega,\,\mathcal{F},\,\{\mathcal{F}_{t}\}_{t\geq0},\,\mathbb{P}) = (\Omega^{0}\times \Omega^{1} \times \Omega^{C},\,\mathcal{F}^{0} \otimes \mathcal{F}^{1} \otimes \mathcal{F}^{C}, \,\{\mathcal{F}^{0}_{t} \otimes \mathcal{F}^{1}_{t} \otimes \mathcal{F}^{C} \}_{t\geq0},\,\mathbb{P}^{0}\times \mathbb{P}^{1} \times \mathbb{P}^{C}). 
\end{equation*}
The standard Brownian motions $W^{0}$ and $B^{0}$ are defined on the complete filtered probability space $(\Omega^{0},\,\mathcal{F}^{0},\,\{\mathcal{F}^{0}_{t}\}_{t\geq0},\,\mathbb{P}^{0})$ with $\mathcal{F}^{0}_{t}$ generated by $W^{0}$, $B^{0}$ and the initial data of the problem. The standard Brownian motions $W^i$ and $B^i$ for $i \geq 1$ are defined on the complete filtered probability space $(\Omega^{1},\,\mathcal{F}^{1},\,\{\mathcal{F}^{1}_{t}\}_{t\geq0},\,\mathbb{P}^{1})$ with $\mathcal{F}^{1}_{t}$ generated by $\{W^{1},\,B^{1},\,W^{2},\,B^{2},\,\ldots\}$ and the initial data of the problem. The coefficient vectors $C_{i}=\left(k_{i},\,\theta_{i},\,\xi_{i},\,r_{i},\,\rho_{1,i},\,\rho_{2,i}\right)$
for $i\in\mathbb{N}$ are defined on the complete probability space $(\Omega^{C},\,\mathcal{F}^{C},\,\mathbb{P}^{C})$ such that $\mathbb{P}^{C}$- almost surely we have $\rho_{1,i},\,\rho_{2,i} \in (-1, 1)$, $\forall i\in\mathbb{N}$. We assume that the Brownian motions $W^{0}$ and $B^{0}$ are generally correlated with $dW^0_t \cdot dB^0_t = \rho_3dt$, while the random sequences $\{W^{1},\,B^{1},\,W^{2},\,B^{2},\,\ldots\}$ and $\{C_{1}, \, C_2, \, \ldots\}$ are assumed to consist of independent terms. Finally, we assume that the initial values $\left\{x^{1},\,\sigma^{1},\,x^{2},\,\sigma^{2},\,\ldots\right\}$ are measurable with respect to the $\sigma$-algebra $\mathcal{F}_{0} = \mathcal{F}^{0}_{0} \otimes \mathcal{F}^{1}_{0} \otimes \mathcal{F}^{C}$ and independent from the Brownian motions $\{W^{0},\,B^{0},\,W^{1},\,B^{1},\,W^{2},\,B^{2},\,\ldots\}$, and they form an exchangeable sequence of two-dimensional random variables $\{(x^{i},\,\sigma^{i})\}_{i=1}^{\infty}$. The exchangeability condition implies that there exists a $\sigma$-algebra $\mathcal{G}$ contained in $\mathcal{F}_{0}$
such that the above two-dimensional initial values are i.i.d given $\mathcal{G}$ (see, for example, \cite{Aldous}) and, without loss of generality, we may assume that $\mathcal{G} = \mathcal{G}^{0} \otimes \{\emptyset, \, \Omega^{1} \} \otimes \{\emptyset, \, \Omega^{C} \}$, for some $\mathcal{G}^{0} \subset \mathcal{F}^{0}_{0}$. This is clear as we can take $\Omega^{0}$ large enough to make the initial vectors measurable with respect to $\mathcal{F}^{0}_{0} \otimes \{\emptyset, \, \Omega^{1} \} \otimes \{\emptyset, \, \Omega^{1} \}$, but allowing for some dependence on $\Omega^{1} \times \Omega^{C}$ by defining them as two-dimensional random variables with respect to the larger $\sigma$-algebra provides some consistency if we decide to restart our system at some $t > 0$.

Under the above assumptions and for each $N\in\mathbb{N}$, we consider the interacting particle system described by equations \eqref{eq:model}
and the corresponding empirical measure processes $v^{N}$ and
$v_{1}^{N}$, which are defined by equations \eqref{eq:1.5} and \eqref{eq:1.6} respectively. We also define the process $v_2^N$ by $v_{2,t}^{N} := v_{t}^{N}-v_{1,t}^{N}$, the restriction of $v_{t}^{N}$
to $\{0\}\times\mathbb{R}$, for all $t\geq0$. These quantities are also defined in \cite{HK17} for the CIR volatility case and it is not hard to check that all the convergence theorem proofs in \cite[Section 2]{HK17} do not depend on the form of the SDE satisfied by the volatility processes, as long as these processes are continuous, strong solutions to that SDE. Moreover, the $\sigma$-algebra $\mathcal{G}$ there plays the role of the $\sigma$-algebra $\mathcal{G}$ here, which can be decomposed as $\mathcal{G} = \mathcal{G}^{0} \otimes \{\emptyset, \, \Omega^{1} \} \otimes \{\emptyset, \, \Omega^{C} \}$. Therefore, working given $\mathcal{G}$ is the same as working given $\mathcal{G}^0$, and under the general stochastic volatility setting we can have the convergence results of in \cite[Section 2]{HK17} as follows:

\begin{thm}\label{thm:2.1}
For each $N\in\mathbb{N}$ and any $t,\,s\geq0$, consider the random measure given by
\[
v_{3,t,s}^{N}=\frac{1}{N}\sum_{i=1}^{N}\delta_{X_{t}^{i},\sigma_{t}^{i},\sigma_{s}^{i}}.
\]
The sequence $v_{3,t,s}^{N}$ of three-dimensional empirical measures
converges weakly to some measure $v_{3,t,s}$ for all $t,\,s\geq0$,
$\mathbb{P}$-almost surely. Moreover, the measure-valued process
$\{v_{3,t,s}:\,t,\,s\geq0\}$ is $\mathbb{P}$-almost surely continuous
in both $t$ and $s$ under the weak topology.
\end{thm}

\begin{cor}\label{cor:2.2}
The sequence $v_{t}^{N}$ of two-dimensional empirical measures given
by \eqref{eq:1.5} converges weakly to some measure $v_{t}$ for all $t\geq0$,
$\mathbb{P}$-almost surely. Moreover, the path $\{v_{t}:\,t\geq0\}$
is $\mathbb{P}$-almost surely continuous under the weak topology.
The measure-valued process $v_{t}$ is the restriction of
$v_{3,t,s}$ to the space of functions which are constant in the third
variable, for any $t\geq0$.
\end{cor}

\begin{thm}\label{thm:2.3}
There exists an $\Omega'\subset\Omega$ with $\mathbb{P}(\Omega')=1$
such that for any $\omega\in\Omega'$, we have $\int_{\mathbb{R}^{3}}f dv_{3,t,s}=\mathbb{E}\left[f\left(X_{t}^{1},\,\sigma_{t}^{1},\,\sigma_{s}^{1}\right)|\,W^{0},\,B^{0},\,\mathcal{G}^0\right]$
for any $t,\,s\geq0$ and any $f\in C_{b}(\mathbb{R}^{3};\,\mathbb{R})$.
\end{thm}

\begin{cor}\label{cor:2.4}
Let $\{v_{t}:\,t\geq0 \}$ be the measure-valued process defined in Corollary~\ref{cor:2.2}. 
There exists an $\Omega'\subset\Omega$ with $\mathbb{P}(\Omega')=1$
such that for any $\omega\in\Omega'$, we have $\int_{\mathbb{R}^{2}}f dv_{t}=\mathbb{E}\left[f\left(X_{t}^{1},\,\sigma_{t}^{1}\right)|\,W^{0},\,B^{0},\,\mathcal{G}^0\right]$
for any $t\geq0$ and for any $f\in C_{b}\left(\mathbb{R}^{3};\,\mathbb{R}\right)$.
\end{cor}

\begin{thm}\label{thm:2.5}
There exists a measure-valued process $\{v_{2,t}:\,t\geq0\}$
and an $\Omega''\subset\Omega'$ with $\mathbb{P}(\Omega'')=1$,
such that for any $\omega\in\Omega''$ we have that $v_{2,t}^{N}\xrightarrow{N \to +\infty} v_{2,t}$
weakly for all $t\geq0$. Moreover, we have $\int_{\mathbb{R}^{2}}f dv_{2,t}=\mathbb{E}\left[f\left(X_{t}^{1},\,\sigma_{t}^{1}\right)\mathbb{I}_{\{T_{1}<t\}}|\,W^{0},\,B^{0},\,\mathcal{G}^0\right]$
for all $t\geq0$ and for all $f\in C_{b}\left(\mathbb{R}^{2};\,\mathbb{R}\right)$.
\end{thm}

By Corollary~\ref{cor:2.2} and Theorem~\ref{thm:2.5} we have the weak convergence 
$v_{1,t}^{N}=v_{t}^{N}-v_{2,t}^{N}\to v_{t}-v_{2,t}=:v_{1,t}$, which holds for all $t\geq0$, $\mathbb{P}$-almost surely. Moreover, by Corollary~\ref{cor:2.4} and Theorem~\ref{thm:2.5} we can write
\begin{eqnarray*}
\int_{\mathbb{R}^{2}}f dv_{1,t} &=& \mathbb{E}\left[f\left(X_{t}^{1},\,\sigma_{t}^{1}\right)\mathbb{I}_{\{T_{1}>t\}}\,|\,W^{0},\,B^{0},\,\mathcal{G}^0\right]  \\
&=& \mathbb{E}\left[\mathbb{E}\left[f\left(X_{t}^{1},\,\sigma_{t}^{1}\right)\mathbb{I}_{\{T_{1}>t\}}\,|\,W^{0},\,B^{0},\,C_{1},\,\mathcal{G}^0\right]\,|\,W^{0},\,B^{0},\,\mathcal{G}^0\right],
\end{eqnarray*}
for any $f\in C_{b}\left(\mathbb{R}^{2};\,\mathbb{R}\right)$ and
$t\geq0$, $\mathbb{P}$-almost surely. Note also that conditional expectations given $\left(W^{0},\,B^{0},\,\mathcal{G}^0\right)$ are defined on the component probability space $(\Omega^{0},\,\mathcal{F}^{0},\,\{\mathcal{F}^{0}_{t}\}_{t\geq0},\,\mathbb{P}^{0})$, while conditional expectations given $\left(W^{0},\,B^{0},\,C_{1},\,\mathcal{G}^0\right)$ are defined on the product space $(\Omega^{0},\,\mathcal{F}^{0},\,\{\mathcal{F}^{0}_{t}\}_{t\geq0},\,\mathbb{P}^{0}) \times (\Omega^{C},\,\mathcal{F}^{C},\,\mathbb{P}^{C})$. Therefore, the next step is to study the process of measures $v_{t,C_{1}}(\cdot)$ defined as
\[
v_{t,C_{1}}\left(\cdot \right)=\mathbb{P}\left(\left(X_{t}^{1},\sigma_{t}^{1}\right) \in \cdot , T_{1}>t \,|W^{0},B^{0},C_{1},\mathcal{G}^0\right),
\]
on the product space $(\Omega^{0},\,\mathcal{F}^{0},\,\{\mathcal{F}^{0}_{t}\}_{t\geq0},\,\mathbb{P}^{0}) \times (\Omega^{C},\,\mathcal{F}^{C},\,\mathbb{P}^{C})$. However, by fixing $\omega^C \in \Omega^C$ and studying $v_{t,C}(\cdot)$ for $C = C_1\left(\omega^C\right)$, we reduce this problem to the case where all coefficient vectors are equal to the deterministic vector $C = \left(k,\,\theta,\,\xi,\,r,\,\rho_{1},\,\rho_{2}\right)$. The behaviour of this measure-valued process, which we denote by $v_{t}(\cdot)$ (suppressing $C$ for simplicity), is given in the following Theorem the proof of which follows exactly the same steps as the proof of \cite[Theorem~2.6]{HK17} and can be found in the Appendix (subsection~\ref{pf2.6}). 

\begin{thm}\label{thm:2.6}
Let $A$ be the two-dimensional differential operator
mapping any smooth function $f:\mathbb{R}^{+}\times\mathbb{R}\rightarrow\mathbb{R}$
to
\begin{eqnarray*}
Af\left(x,\,y\right) &=& \left(r-\frac{h^{2}\left(y\right)}{2}\right) f_{x}\left(x,\,y\right)+k\left(\theta-y\right)f_{y}
\left(x,\,y\right)+\frac{1}{2}h^{2}\left(y\right)f_{xx}\left(x,\,y\right) \\
& & \qquad +\frac{1}{2}\xi^{2}q^2(y)f_{yy}\left(x,\,y\right)+\xi\rho_{3}\rho_{1}\rho_{2}h(y)q(y)f_{xy}\left(x,\,y\right)
\end{eqnarray*}
for all $\left(x,\,y\right)\in\mathbb{R}^{+}\times\mathbb{R}$. Then, the measure-valued stochastic process $v_{t}\left(\cdot\right)$ defined on $(\Omega^{0},\,\mathcal{F}^{0},\,\{\mathcal{F}^{0}_{t}\}_{t\geq0},\,\mathbb{P}^{0})$ satisfies the following weak form SPDE
\begin{eqnarray*}
\int_{\mathbb{R}^{2}}f\left(x,\,y\right) dv_{t}\left(x,\,y\right)&=&\int_{\mathbb{R}^{2}}f\left(x,\,y\right) dv_{0}\left(x,\,y\right) \\
& & +\int_{0}^{t}\int_{\mathbb{R}^{2}}Af\left(x,\,y\right) dv_{s}\left(x,\,y\right)ds \\
& & +\rho_{1}\int_{0}^{t}\int_{\mathbb{R}^{2}}h\left(y\right) f_{x}\left(x,\,y\right) dv_{s}\left(x,\,y\right)dW_{s}^{0} \\
& & +\rho_{2}\xi\int_{0}^{t}\int_{\mathbb{R}^{2}}q(y) f_{y}\left(x,\,y\right) dv_{s}\left(x,\,y\right)dB_{s}^{0},
\end{eqnarray*}
for all $t\geq0$ and any $f\in C_{0}^{test}=\left\{ g\in C_{b}^{2}\left(\mathbb{R}^{+}\times\mathbb{R}\right):\,g\left(0,\,y\right)=0,\:\forall\,y\in\mathbb{R}\right\} $.
\end{thm}

\section{Analysis of the volatility}

As in \cite{HK17}, after showing convergence of the empirical measure process to the probabilistic solution of an SPDE, we establish the existence of a regular density for that solution by following a similar approach. Therefore, our first step is to show that a regular density exists $\mathbb{P}^{0}$-almost surely for the volatility component of the solution, i.e the 1-dimensional distribution-valued process whose value at time $t$ maps any suitable function $f: \mathbb{R} \mapsto \mathbb{R}$ to $\mathbb{E}\left[f(\sigma_{t})\,|\,B^{0},\,\mathcal{G}^0\right]$, where $\sigma$ is a general mean-reverting volatility process driven by a combination of $B^{0}$ and $B^1$, that is a process satisfying
\begin{equation}\label{eq:3.1}
d\sigma_{t}=k(\theta-\sigma_{t})dt+ \xi q\left(\sigma_{t}\right)\sqrt{1-\rho_{2}^{2}}dB_{t}^{1}+\xi q\left(\sigma_{t}\right)\rho_{2}dB_{t}^{0},\;t\geq0
\end{equation}
for some suitable function $q$, with $\sigma_0$ being $F_0^0 \times F_0^1$-measurable. The assumption we make about $q$ is the following:

\begin{ass}\label{ass1}
The function $q$ belongs to the space $C^3\left(\mathbb{R}\right)$, is bounded above and strictly away from 0, and has bounded $\mathcal{O}\left(\frac{1}{|x|}\right)$ (as the argument $x$ tends to $\pm \infty$) derivatives up to third order.
\end{ass}

\begin{rem}\label{rem3}
The class of functions $q$ satisfying Assumption~\ref{ass1} clearly contains all functions which are identically equal to a positive constant, making the Theorem applicable to an Ornstein-Uhlenbeck volatility setting. However, it is actually a much bigger class, which contains all positive $C^3$ functions behaving like a positive constant for large $x$. We can also construct functions belonging to that class which are equal to the square root function in some compact subinterval of $\left(0, \, +\infty\right)$, which could be used to obtain effective approximations to the CIR model for which we were unable to obtain very good results in \cite{HK17,ERR}.
\end{rem}

If we assume that $\rho_2 \in (-1, \, 1)$, the answer to our question is given in the next theorem
\begin{thm}\label{thm:3.1}
Suppose that $q$ satisfies Assumption~\ref{ass1} and $\sigma_{0}$ is a random variable in $L^{p}\left(\Omega^0,\,\mathcal{F}_{0}^0,\,\mathbb{P}^0\right)$
for all $p \geq 0$. Then $\mathbb{P}^0$ - almost surely, the conditional probability measure $\mathbb{P}(\sigma_{t}\in A\,|\,B^{0},\,\mathcal{G}^0), \, A \subset \mathbb{R}$ has a continuous density $p_{t}(y\,|\,B^{0},\,\mathcal{G}^0), \, y \in \mathbb{R} $ for all $t > 0$. Moreover, for any $T>0$, if we define
\[
M(t) := \esssup_{y \in \mathbb{R}, \, \omega^0 \in \Omega^0}p_{t}(y\,|\,B^{0},\,\mathcal{G}^0)\left(\omega^0\right),
\]
and also
\[
M_{B^{0},\, \mathcal{G}^0}^{\alpha}(t) := \sup_{y \in \mathbb{R}}\left(|y|^{\alpha}p_{t}(y\,|\,B^{0},\,\mathcal{G}^0)\right), 
\]
for all $0 \leq t \leq T$ and any $\alpha \geq 0$, we have both $M \in L^{1}\left(\left[0, \, T\right]\right)$ and $M_{B^{0},\, \mathcal{G}^0}^{\alpha} \in L^{p}\left(\Omega^0 \times \left[0, \, T\right]\right)$ for any $1 \leq p < 2$.
\end{thm}

Observe that for each $t$, $M(t)$ is a deterministic bound for the density given $B^0$, something we did not have for the volatility density bounds obtained in \cite{HK17,ERR}, and as we have already mentioned, this will play a decisive role in obtaining the desired results. The proof of the above theorem, which is based on a few results that involve the notion of Malliavin differentiability, can be found in the Appendix (subsection~\ref{pf3.3}). We refer to \cite{Nualart} for the basics of Malliavin
calculus, and to \cite{NUZA} for an extension called the partial Malliavin calculus which applies to conditional probability measures. In this section we are working under the conditional probability measure $\mathbb{P}(\cdot\,|\,B^{0},\,\mathcal{G}^0)$, which means that partial Malliavin calculus is the natural tool for proving our results. However, it is not hard to see that in our setting we can $\mathbb{P}^0$-almost surely fix an $\omega^0 \in \Omega^0$, and work under $\mathbb{P}(\sigma_{t} \in \cdot\,|\,B^{0},\,\mathcal{G}^0) = \mathbb{P}^{1}(\{\omega^1 \in \Omega^1: \sigma_{t}(\omega^1, \, \omega^0) \in \cdot \})$ for that given $\omega^0$, where ordinary Malliavin calculus can be used with respect to the Brownian motion $B^1$. It is not hard to check that the Malliavin derivative with respect to $B^1$ for any continuous path of $B^{0}$, coincides with the partial Malliavin derivative in the sense of \cite{NUZA}, when the random subspace $K\left(\omega\right)$ is the orthogonal complement of the space generated by the Malliavin derivatives of $B_{r}^{0}$ with respect to $\left(B^{0}, B^1\right)$, for $r \in \mathbb{Q}^{+}$. 

\section{The two-dimensional density and the initial-boundary value problem}

We proceed now to the proof of the existence of a regular density for the full two-dimensional distribution-valued process $v_{t}$ which satisfies the SPDE given in Theorem~\ref{thm:2.6}. First, given a Hilbert space $H$, we denote by $L^{2}\left(\left(\Omega^0, \, \mathcal{F}^0, \, \mathbb{P}^0\right)\times\left[0,\,T\right];\,H\right)$
the space of $H$- valued stochastic processes defined on the probability space $\left(\Omega^0, \, \mathcal{F}^0, \, \mathbb{P}^0\right)$ which are $L^2$-integrable and adapted to the filtration generated by the two-dimensional Brownian motion $(W^{0},\,B^{0})$ and by $\mathcal{G}^0$. We will use some stronger versions of \cite[Theorem 4.1]{HK17} and \cite[Lemma E2.2]{ERR}, and we will work in the function spaces 
\[
L_{\alpha}=L^{2}\left(\left(\Omega^0, \, \mathcal{F}^0, \, \mathbb{P}^0\right)\times\left[0,\,T\right];\,L_{|y|^{\alpha}}^{2}\left(\mathbb{R}^{+}\times\mathbb{R}\right)\right)
\]
for $\alpha \geq0$ and
\[
H_{0}=L^{2}\left(\left(\Omega^0, \, \mathcal{F}^0, \, \mathbb{P}^0\right)\times\left[0,\,T\right];\,H_{0, w^2(x)}^{1}\left(\mathbb{R}^{+}\right)\times L^{2}\left(\mathbb{R}\right)\right),
\]
for $w(x) = \min\left\{1,\sqrt{x}\right\}$, which were defined in \cite{ERR} as well. As in that paper, $L_{g(y)}^{2}$ stands for the weighted $L^{2}$ space with weight function $\left\{ g(y):y \in \mathbb{R}\right\}$, $H_{0, g(x)}^{1}\left(\mathbb{R}^{+}\right)$ stands for the weighted $H_{0}^{1}\left(\mathbb{R}^{+}\right)$ space with weight function $\left\{ g(x):x\geq0\right\}$ in the $L^2$ norm of the weak derivative, and any function $u'$ belonging to the $H_{0}$ space will also satisfy $\displaystyle{\lim_{x \rightarrow 0^{+}}\left\Vert u'(\cdot, \, x, \, \cdot) \right\Vert_{L^2\left(\Omega^0 \times \left[0, \, T\right] \times \mathbb{R}\right)} = 0}$. The last condition will be the boundary condition in our stochastic initial-boundary value problem, which cannot be defined via a trace operator due to $\int_0^{\epsilon}w^{-2}(x)dx$ being infinite for all $\epsilon > 0$. 

Since the results in \cite{ERR} were established under the condition $\rho_3 = 0$, it was not that important for the volatility path in \cite[Theorem 4.1]{HK17} to depend on $W^0$. However, now that $W^0$ and $B^0$ are allowed to be correlated, the following extension is needed:

\begin{thm}\label{thm:4.1}
For an $\mathcal{F}_t$-adapted stochastic process $\sigma = \{\sigma_{t}\}_{t\geq0}$ which is independent from $W^1$, has continuous paths and takes values in a compact subinterval of $\left(0, \, +\infty\right)$, let $\{X_{t}^{\sigma}:\,t\geq0\}$ be the solution to the stopped SDE 
\begin{equation}\label{eq:4.1}
\begin{cases}
dX_{t}^{\sigma}=\left(r-\frac{\sigma_{t}}{2}\right)dt+\sqrt{\sigma_{t}}\sqrt{1-\rho_{1}^{2}}dW_{t}^{1} 
+\sqrt{\sigma_{t}}\rho_{1}dW_{t}^{0}, & 0\leq t\leq\tau,\\
X_{t}^{\sigma}=0, & t\ge\tau,
\end{cases}
\end{equation}
for $\tau=\inf\{t\geq0:\,X_{t}^{\sigma}=0\}$, under the initial
condition $X_{0}^{\sigma}=X_{0}$, where $X_{0}$ is a continuous random variable. Suppose that given $\mathcal{G}^0$ and $\sigma_0$, $X_{0}$ has an $L^{2}$-integrable density $u_{0}=u_{0}(\cdot\,| \,\mathcal{G}^0, \,\sigma_0)$. Finally, let
$\{V_{t}^{\sigma}:\,t\geq0\}$ be the measure-valued process
\[
V_{t}^{\sigma}(A)=\mathbb{P}\left(X_{t}^{\sigma}\in A,\tau\geq t\,|\,W^{0},\,\mathcal{G}^0, \sigma\right)
\]
for any Borel set $A\subset\left(0,\,\infty\right)$. Then almost
surely, the following are true for all $T>0$;
\begin{enumerate}
\item  $V_{t}^{\sigma}$ possesses a density $u^{\sigma}(t, \, \cdot) = u^{\sigma}\left(t,\,\cdot\,;\,W^{0},\,\mathcal{G}^0\right)$
for all $0\leq t\leq T$, which is the unique solution to the SPDE
\begin{equation}\label{eq:4.2}
du^{\sigma}(t,\,x)=-\left(r-\frac{\sigma_{t}}{2}\right)u_{x}^{\sigma}(t,\,x)dt+\frac{\sigma_{t}}{2}
u_{xx}^{\sigma}(t,\,x)dt-\sqrt{\sigma_{t}}\rho_{1}u_{x}^{\sigma}(t,\,x)dW_{t}^{0}
\end{equation}
in $L^{2}\left(\left(\Omega^0, \, \mathcal{F}^0, \, \mathbb{P}^0\right)\times\left[0,\,T\right];\,H_{0}^{1}\left(\mathbb{R}^{+}\right)\right)$
under the initial condition $u^{\sigma}(0,\cdot)=u_{0}$. 

\item For all $0\leq t\leq T$, the following identity holds
\begin{equation}\label{eq:4.3}
\left\Vert u^{\sigma}(t,\,\cdot)\right\Vert_{L^{2}(\mathbb{R}^{+})}^{2}+\left(1-\rho_{1}^{2}\right)\int_{0}^{t}\sigma_{s}
\left\Vert u_{x}^{\sigma}(s,\,\cdot)\right\Vert_{L^{2}(\mathbb{R}^{+})}^{2}ds=\left\Vert u_{0}\right\Vert_{L^{2}(\mathbb{R}^{+})}^{2}.
\end{equation}
\end{enumerate}
\end{thm}

\begin{proof}[Proof (outline)]
We will describe how the proof of \cite[Theorem 4.1]{HK17} needs to be modified to give this extension. First, we need the corresponding extension of \cite[Lemma 4.2]{HK17}, which was the main tool for proving \cite[Theorem 4.1]{HK17}:
\begin{lem}{\label{convergence}}
Let $W = \sqrt{1-\rho_{1}^{2}}dW_{t}^{1} + \rho_{1}dW_{t}^{0}$, and let also $\{\sigma_{t}^{0}:\,0\leq t\leq T\}$ be a positive and continuous stochastic process which is approximated from above by a pointwise decreasing sequence $\left(\{\sigma_{t}^{m}:\,0\leq t\leq T\}\right)_{m\in\mathbb{N}}$ of $L^2\left(\Omega \times \left[0,\,T\right]\right)$-integrable, $\mathcal{F}_t$-adapted and $W^1$-independent processes, with the convergence holding for each $t \in \left[0, \, T\right]$ both in $L^2\left(\Omega\right)$ and for all $\omega$ in some $\tilde{\Omega}$ which does not depend on $t$ and has probability $1$. For any $m \in \mathbb{N}\cup \{0\}$, denote by $X^{m}$ the stopped Ito process given by
\begin{equation}\label{eq:4.4}
\begin{cases}
dX_{t}^{m}=\left(r-\frac{\sigma_{t}^{m}}{2}\right)dt+\sqrt{\sigma_{t}^{m}}dW_{t}, & 0\leq t\leq\tau^{m},\\
X_{t}^{m}=0, & t > \tau^{m},
\end{cases}
\end{equation}
where $\tau^{m}=\inf\{t\geq0:\,X_{t}^{m}=0\}$, with the initial condition
\[
X_{0}^{m}=\max\left\{x_{0}-l_{m},\,\frac{x_{0}}{2}\right\},
\]
for $x_0 \geq 0$ and 
\[l_{m}=\left(\left\Vert \sqrt{\sigma_{.}^{m}}-\sqrt{\sigma_{.}^{0}}\right\Vert _{L^{2}\left(\Omega \times \left[0,\,T\right]\right)}\right)^{1/2}\] 
Then, for a sequence $\{m_{k}:\,k\in\mathbb{N}\}\subset\mathbb{N}$,
we have almost surely: $X_{t}^{m_{k}}\rightarrow X_{t}^{0}$ uniformly in $\left[0,\,T\right]$.
\end{lem}
\noindent The proof of the above is identical to that of \cite[Lemma 4.2]{HK17}, as all the arguments remain valid when the deterministic volatility paths are replaced with $L^2\left(\Omega \times \left[0,\,T\right]\right)$-integrable, $\mathcal{F}_t$-adapted and $W^1$-independent processes. 

Now we modify the proofs of 1 and 2 of \cite[Theorem 4.1]{HK17} to obtain 1 and 2 of this extension respectively. The proof of 1 does not need any modification beyond replacing $(\Omega, \mathcal{F}, \mathbb{P})$ with $(\Omega^0, \mathcal{F}^0, \mathbb{P}^0)$, as the same arguments apply when $\sigma$ is non-deterministic. For the proof of 2 we need to repeat steps 1, 2 and 3 with the approximating volatilies being $\mathcal{F}_t$-adapted and $W^1$-independent processes which are constant between consecutive terms of sequences of $W^1$-independent $\mathcal{F}_t$ - stopping times. For step 1 there is nothing to change since an $\mathcal{F}_t$-adapted constant process is always deterministic given its initial value, so the results of \cite{BHHJR, HL16} are applicable again. For step 2, the piecewise constant path $\sigma$ is determined by a sequence $\{t_m\}_{m=1}^{\infty}$ of $W^1$-independent $\mathcal{F}_t$ - stopping times and the sequence of random variables $\{\sigma_{t_m}\}_{m=1}^{\infty}$. Denoting $\sigma^{t_m} = \{\sigma_{t \wedge t_m}\}_{t \geq 0}$, for $t \leq t_{n+1}$ we can write
\begin{align}\label{stopped2}
V_t^{\sigma}(A) 
&= \mathbb{P}\left(X_{t}^{\sigma^{t_n}} \in A, \, t \leq \tau \, \bigg| \, W^0, \mathcal{G}^0, \{t_m\}_{m=1}^{n}, \{\sigma_{t_m}\}_{m=1}^{n}\right)
\end{align}
so a density for $V_t^{\sigma}(\cdot)$ that satisfies \eqref{eq:4.2} on $[0, t_n]$ can be extended to $[0, t_{n+1}]$ by using step 1 along with the strong Markov property for the right hand side of \eqref{stopped2} on $[t_n, \infty)$ and then stopping at $t_{n+1}$. Then, for step 3, the only things we change are that we replace \cite[Lemma~4.2]{HK17} with its extension obtained here (Lemma~\ref{convergence} above), we condition on $(W^0, \mathcal{G}^0, (\sigma^{m_k})_{k \in \mathbb{N}})$ instead of $(W^0, \mathcal{G})$ in all the conditional probabilities and densities that are considered in the original proof (so the set $B$ in the last display is taken in the $\sigma$-algebra generated by $(W^0, \mathcal{G}^0, (\sigma^{m_k})_{k \in \mathbb{N}})$ instead of $(W^0, \mathcal{G})$), and we have the second term in \cite[Equation (4.5)]{HK17} with $\sigma_s^{m_k}$ inside the norm since $\sigma^{m_k}$ is no longer a deterministic path. Because of the latter, we need to recall that $m= \min_{0\leq s\leq T}\{\sigma_{s}\}$ has a deterministic positive lower bound in order to derive the weak convergence of $u^{n_k}$ in $L^{2}\left(\Omega\times\left[0,\,T\right];\,H_{0}^{1}\left(\mathbb{R}^{+}\right)\right)$. The rest of the arguments remain unchanged, and even the clarification in \cite[Appendix~A]{ERR} is still valid. 

Finally, we need to justify the existence of the approximating sequence before applying Lemma 4.2, which is not that obvious when the paths are non-deterministic. For the later, we need to construct a pointwise decreasing sequence $\{\{\sigma_t^{m}\}_{t \in \left[0, \, T\right]}\}_{m \in \mathbb{N}}$ of piecewise constant, $\mathcal{F}_t$-adapted, $W^1$-independent and $L^2\left(\Omega \times \left[0,\,T\right]\right)$-integrable processes, which approximate $\sigma = \{\sigma_t\}_{t \geq 0}$ from above. This sequence can be constructed as follows:
For each $m \in \mathbb{N}^+$ we define $t_0^m = 0$ and then
\begin{eqnarray}
t_{k+1}^m = \inf\left\{t > t_k^m : \sigma_{t_k^m} + \frac{1}{m} = \sigma_t\right\} \wedge \left(t_k^m + \frac{1}{m}\right) \wedge T \nonumber 
\end{eqnarray}
for any $k \in \mathbb{N} = \{0, 1, 2, \ldots\}$, where we see that almost surely we have $t_k^m = T$ for large enough $k$, as otherwise the increasing (in $k$) sequence $t_k^m$ would have a limit $\tilde{T} \leq T$, and taking limits in $\sigma_{t_k^m} + \frac{1}{m} = \sigma_{t_k^{k+1}}$ would yield $\sigma_{\tilde{T}} = \frac{1}{m} + \sigma_{\tilde{T}} \Rightarrow \frac{1}{m} = 0$, a contradiction. Next we define
\begin{eqnarray}\label{pcaf}
\tilde{\sigma}_t^m = \sigma_{t_k^m} + \frac{1}{m}, \qquad t \in \left[t_k^m, \, t_{k+1}^m\right).
\end{eqnarray}
for any $k \in \mathbb{N}$ with $t_k^m < T$, which is constant between any two consecutive terms of a sequence of $W^1$-independent stopping times and clearly also $W^1$-independent. Finally we define $\tilde{\sigma}_t^{0} = 2C$ and $\displaystyle{\sigma_t^m = \inf_{0 \leq m' \leq m}\tilde{\sigma}_t^{m'}}$ for all $t \leq T$ and $m \in \mathbb{N}^{+}$ to achieve the required monotonicity and integrability. It is not hard to check that the last sequence of processes has all the required properties, and by using \eqref{pcaf} and the fact that $|t_{k+1}^m - t_k^m| \leq \frac{1}{m}$ for all $k$ and $m$ we can also show that almost surely (and in $L^2\left(\Omega\right)$ by dominated convergence) it converges pointwise to $\{\sigma_t\}_{t \in \left[0, \, T\right]}$. 
\end{proof}
\noindent We are now ready to establish the existence of a regular density for $v_{t}$.

\begin{thm}\label{thm:4.2}
Suppose that the function $h$ is positive, bounded and continuous. Suppose also that, given $\mathcal{G}^0$ and $\sigma_0^1$, $X^{1}_{0}$ has an $L^{2}$-integrable density $u_0 = u_{0}(\cdot|\mathcal{G}^0, \sigma_0^1)$ on $\mathbb{R}^{+}$ such that $\mathbb{E}\left[ \left \Vert w(\cdot)\left(u_0\right)_{x}(\cdot)\right \Vert_{L^{2}(\mathbb{R}^{+})}^{2}\right] < \infty$ and $\mathbb{E}\left[\left \Vert u_{0} \right \Vert_{L^{2}\left(\mathbb{R}^{+}\right)}^{2} \, | \, \mathcal{G}^0 \right]
\in L^{p'}\left(\Omega^0\right)$ for some $p' > 2$. Finally, suppose that the assumptions of Theorem~\ref{thm:3.1} are satisfied for the function $q$ and for $\sigma_{0} = \sigma^{1}_{0}$.
Then, the measure-valued stochastic process $v_{t}$ has a two-dimensional density $u = u(t,\,\cdot; \,W^{0},\,B^{0},\,\mathcal{G}^0)$
belonging to the space $L_{\alpha}$ for all $\alpha \geq 0$, and also to the space $H_{0}$.
\end{thm}

\begin{proof}
Let $f$ be a smooth function, compactly supported
in $\mathbb{R}^{2}$, such that $f$ vanishes on the $y$ - axis. Then by Theorem~\ref{thm:3.1} we have
\begin{eqnarray}
v_{t}\left(f \right)&=& \mathbb{E}\left[f\left(X_{t}^{1}, \, \sigma_{t}^{1}\right)\mathbb{I}_{\{T_{1}\geq t\}}\,|W^{0},B^{0},\mathcal{G}^0\right] \nonumber \\
&=&\mathbb{E}\left[\mathbb{E}\left[f\left(X_{t}^{1}, \, \sigma_{t}^{1}\right)\mathbb{I}_{\{T_{1}\geq t\}}\,|W^{0},\sigma_{t}^{1},B^{0},\mathcal{G}^0\right]\,|W^{0},B^{0},\mathcal{G}^0\right]
\nonumber \\
&=&\int_{\mathbb{R}}\mathbb{E}\left[f\left(X_{t}^{1}, \, y\right)\mathbb{I}_{\{T_{1}\geq t\}}\,|W^{0},\sigma_{t}^{1}=y,B^{0},\mathcal{G}^0\right]p_{t}\left(y|B^{0},\mathcal{G}^0\right)dy. 
\label{eq:4.9}
\end{eqnarray}
Next we have
\begin{eqnarray}
& & \mathbb{E}\left[f\left(X_{t}^{1}, \, y\right)\mathbb{I}_{\{T_{1}\geq t\}}|W^{0},
\sigma_{t}^{1}=y,B^{0},\mathcal{G}^0\right] \nonumber \\
& & \qquad =\mathbb{E}\left[\mathbb{E}\left[f\left(X_{t}^{1}, \, y\right)\mathbb{I}_{\{T_{1}\geq t\}}|
W^{0},\mathcal{G}^0, \sigma^1\right]|W^{0},\sigma_{t}^{1}=y,B^{0},
\mathcal{G}^0\right] \nonumber \\
& & \qquad =\mathbb{E}\left[\int_{\mathbb{R}^+}f(x, \, y)u^{h\left(\sigma^1\right)}\left(t,x; W^{0},\mathcal{G}^0\right)dx\Big|W^{0},\sigma_{t}^{1}=y,B^{0},\mathcal{G}^0\right] 
\nonumber \\
& & \qquad =\int_{\mathbb{R}^+}f(x, \, y)\mathbb{E}\left[u^{h\left(\sigma^1\right)}\left(t,x; W^{0},\mathcal{G}^0\right) \Big|W^{0},\sigma_{t}^{1}=y,B^{0},\mathcal{G}^0\right]dx,  \label{eq:4.10}
\end{eqnarray}
where we are using the notation of Theorem~\ref{thm:4.1}, with $h\left(\sigma^1\right)$ being the path $(h\left(\sigma_t^1\right))_{t \geq 0}$. Substituting \eqref{eq:4.10} in \eqref{eq:4.9} 
we find that the desired density exists and is given by
\begin{eqnarray}
&&u\left(t,x,y ; W^{0},B^{0},\mathcal{G}^0\right) \nonumber \\
&& \qquad \qquad =p_{t}\left(y|B^{0},\mathcal{G}^0\right)
\mathbb{E}\left[u^{h\left(\sigma^1\right)}\left(t,x; W^{0},\mathcal{G}^0\right)\,|W^{0},\sigma_{t}^{1}=
y,B^{0},\mathcal{G}^0\right] \nonumber 
\end{eqnarray}
in $\mathbb{R}^{+}\times\mathbb{R}$. Then, by the Cauchy-Schwarz inequality we have
\begin{eqnarray*}
& & w^2(x)\left(\frac{\partial u}{\partial x}\left(t,\,x,\,y;\,W^{0},\,B^{0},\,\mathcal{G}^0\right)\right)^{2} \\
& & \qquad \leq  M(t) p_{t}\left(y\,|\,B^{0},\,\mathcal{G}^0\right) \\
& & \qquad \qquad \times w^2(x)\mathbb{E}^{2}\left[\frac{\partial}{\partial x}u^{h\left(\sigma^1\right)}\left(t,x; W^{0},\mathcal{G}^0\right)\,\Big|\,W^{0},\,\sigma_{t}^{1}=y,\,B^{0},\,\mathcal{G}^0\right] \\
& & \qquad \leq  M(t)p_{t}\left(y\,|\,B^{0},\,\mathcal{G}^0\right) \nonumber \\
& & \qquad \qquad \times \mathbb{E}\left[w^2(x)\left(\frac{\partial}{\partial x}u^{h\left(\sigma^1\right)}\left(t,x; W^{0},\mathcal{G}^0\right)\right)^2\,\bigg|\,W^{0},\,\sigma_{t}^{1}=y,\,B^{0},\,\mathcal{G}^0\right],
\end{eqnarray*}
where $\displaystyle{M(t) = \esssup_{y \in \mathbb{R}, \, \omega^0 \in \Omega^0}p_{\cdot}(y\,|\,B^{0},\,\mathcal{G}^0)\left(\omega^0\right)\in L^{1}\left(\left[0, \, T\right]\right)}$
(defined in Theorem~\ref{thm:3.1}). Integrating the above
in $y$ and using the law of total expectation, we obtain
\begin{eqnarray*}
& & \int_{\mathbb{R}}w^2(x)\left(\frac{\partial u}{\partial x}\left(t,\,x,\,y;\,W^{0},\,B^{0},\,\mathcal{G}^0\right)\right)^{2}dy \\
& & \qquad \qquad \leq M(t)\times\mathbb{E}\left[w^2(x)\left(\frac{\partial}{\partial x}u^{h\left(\sigma^1\right)}\left(t,x; W^{0},\mathcal{G}^0\right)\right)^2\,\bigg|\,W^{0},\,B^{0},\,\mathcal{G}^0\right]
\end{eqnarray*}
and thus
\begin{eqnarray*}
& & \mathbb{E}\left[\int_{0}^{T}\int_{\mathbb{R}^+}\int_{\mathbb{R}} w^2(x) 
\left(\frac{\partial u}{\partial x}\left(t,\,x,\,y,\,W^{0},\,B^{0},\,
\mathcal{G}^0\right)\right)^{2}dydxdt\right]
\\
& & \quad \leq\mathbb{E}\Bigg[\int_{0}^{T}M(t) \nonumber \\
& & \qquad \qquad \times \int_{\mathbb{R}^{+}}
\mathbb{E}\left[w^2(x)\left(\frac{\partial}{\partial x}u^{h\left(\sigma^1\right)}\left(t,x; W^{0},\mathcal{G}^0\right)\right)^2\,|
W^{0},B^{0},\mathcal{G}^0\right]dxdt\Bigg] \\
& & \quad = \int_{0}^{T}M(t)\mathbb{E}\left[\int_{\mathbb{R}^{+}}w^2(x)\left(\frac{\partial}{\partial x}u^{h\left(\sigma^1\right)}\left(t,x; W^{0},\mathcal{G}^0\right)\right)^2dx\right]dt \\
& & \quad \leq \int_{0}^{T}M(t)dt\mathbb{E}\left[\sup_{0 \leq s \leq T}\int_{\mathbb{R}^{+}}w^2(x)\left(\frac{\partial}{\partial x}u^{h\left(\sigma^1\right)}\left(t,x; W^{0},\mathcal{G}^0\right)\right)^2dx\right]
\\
& & \quad \leq M'e^{M'T}\int_{0}^{T}M(t)dt\mathbb{E}\left[ \left \Vert w(\cdot)\left(u_0\right)_{x}(\cdot)\right \Vert_{L^{2}(\mathbb{R}^{+})}^{2}\right] \\
& & \quad < \infty
\end{eqnarray*}
by Tonelli's Theorem, Theorem~\ref{thm:3.1} and Lemma~E2.2 from \cite{ERR} (where the constant $M'$ is defined). The reason we can take $M'$ to be deterministic is that $h$ has a deterministic (positive) lower bound and a deterministic upper bound. This is the estimate for the $x$-derivative of
$u$. To obtain the weighted integrability for the density, we use the Cauchy-Schwarz inequality, Tonelli's Theorem and 2 of Theorem~\ref{thm:4.1}, so for any $\alpha \geq 0$ we have
\begin{eqnarray*}
& & \int_{\mathbb{R}^{+}}\int_{\mathbb{R}^{+}}|y|^{a}\left(u\left(t,\,x,\,y;\,W^{0},\,B^{0},\,\mathcal{G}^0\right)\right)^{2}dydx \\
& & \qquad \leq \int_{\mathbb{R}^{+}}\int_{\mathbb{R}^{+}}M_{B^{0},\, \mathcal{G}^0}^{\alpha}(t) p_{t}\left(y\,|\,B^{0},\,\mathcal{G}^0\right) \\
& & \qquad \qquad \times \mathbb{E}^{2}\left[u^{h\left(\sigma^1\right)}\left(t,x; W^{0},\mathcal{G}^0\right)\,|\,W^{0},\,\sigma_{t}^{1}=y,\,B^{0},\,\mathcal{G}^0\right]dydx \\
& & \qquad \leq  M_{B^{0},\, \mathcal{G}^0}^{\alpha}(t)\int_{\mathbb{R}^{+}}\int_{\mathbb{R}^{+}}p_{t}\left(y\,|\,B^{0},\,\mathcal{G}^0\right) \nonumber \\
& & \qquad \qquad \times \mathbb{E}\left[\left(u^{h\left(\sigma^1\right)}\left(t,x; W^{0},\mathcal{G}^0\right)\right)^2\,\Big|\,W^{0},\,\sigma_{t}^{1}=y,\,B^{0},\,\mathcal{G}^0\right]dydx \\
& & \qquad = M_{B^{0},\, \mathcal{G}^0}^{\alpha}(t)\int_{\mathbb{R}^{+}}\mathbb{E}\left[\left(u^{h\left(\sigma^1\right)}\left(t,x; W^{0},\mathcal{G}^0\right)\right)^2\,\Big|\,W^{0},\,B^{0},\,\mathcal{G}^0\right]dx
\\
& & \qquad = M_{B^{0},\, \mathcal{G}^0}^{\alpha}(t)\mathbb{E}\left[\int_{\mathbb{R}^{+}}\left(u^{h\left(\sigma^1\right)}\left(t,x; W^{0},\mathcal{G}^0\right)\right)^2dx\,\Big|\,W^{0},\,B^{0},\,\mathcal{G}^0\right] \\
& & \qquad \leq M_{B^{0},\, \mathcal{G}^0}^{\alpha}(t)\mathbb{E}\left[ \left \Vert u_0(\cdot)\right \Vert_{L^{2}(\mathbb{R}^{+})}^{2}\,|\,\mathcal{G}^0\right] 
\end{eqnarray*}
where $\displaystyle{M_{B^{0},\, \mathcal{G}^0}^{\alpha}(\cdot) = \sup_{y \in \mathbb{R}}\left(|y|^{\alpha}p_{\cdot}(y\,|\,B^{0},\,\mathcal{G}^0)\right) \in L^{p}\left(\Omega^0 \times \left[0, \, T\right]\right)}$
for any $p < 2$ (by Theorem~\ref{thm:3.1}). Taking then $p' > 2$ sufficiently close to $2$ and $p < 2$ such that $\frac{1}{p} + \frac{1}{p'} = 1$, by the above and by H\"{o}lder's inequality we have
\begin{eqnarray*}
& & \mathbb{E}\left[\int_{0}^{T}\int_{\mathbb{R}^{+}}\int_{\mathbb{R}^{+}}|y|^{a}\left(u\left(t,\,x,\,y;\,W^{0},\,B^{0},\,\mathcal{G}^0\right)\right)^{2}dydxdt\right] \\
& & \qquad \leq \mathbb{E}\left[\int_{0}^{T}M_{B^{0},\, \mathcal{G}^0}^{\alpha}(t)\mathbb{E}\left[ \left \Vert u_0(\cdot)\right \Vert_{L^{2}(\mathbb{R}^{+})}^{2}\,|\,\mathcal{G}^0\right]dt\right] \\
& & \qquad \leq \mathbb{E}^{\frac{1}{p}}\left[\left(\int_{0}^{T}M_{B^{0},\, \mathcal{G}^0}^{\alpha}(t)dt\right)^p\right] \mathbb{E}^{\frac{1}{p'}}\left[\mathbb{E}^{p'}\left[ \left \Vert u_0(\cdot)\right \Vert_{L^{2}(\mathbb{R}^{+})}^{2}\,|\,\mathcal{G}^0\right]\right] \\
& & \qquad \leq T^{\frac{1}{p'}}\mathbb{E}^{\frac{1}{p}}\left[\int_{0}^{T}\left(M_{B^{0},\, \mathcal{G}^0}^{\alpha}(t)\right)^pdt\right] \mathbb{E}^{\frac{1}{p'}}\left[\mathbb{E}^{p'}\left[ \left \Vert u_0(\cdot)\right \Vert_{L^{2}(\mathbb{R}^{+})}^{2}\,|\,\mathcal{G}^0\right]\right] \\
& & \qquad < \infty,
\end{eqnarray*}
which implies that the density belongs to the space $L_{\alpha}$ for any $\alpha \geq 0$. Finally, we need to show that $\displaystyle{\lim_{x \rightarrow 0^{+}}\left\Vert u(\cdot, \, x, \, \cdot) \right\Vert_{L^2\left(\Omega^0 \times \left[0, \, T\right] \times \mathbb{R}\right)} = 0}$, which follows from the estimate
\begin{eqnarray*}
& & \mathbb{E}\left[\int_{\mathbb{R}^{+}}\int_{\mathbb{R}^{+}}\left(u\left(t,\,x,\,y;\,W^{0},\,B^{0},\,\mathcal{G}^0\right)\right)^{2}dydt\right] \\
& & \qquad \leq \mathbb{E}\Bigg[\int_{\mathbb{R}^{+}}\int_{\mathbb{R}^{+}}M(t) p_{t}\left(y\,|\,B^{0},\,\mathcal{G}^0\right) \\
& & \qquad \qquad \qquad \qquad \times \mathbb{E}^{2}\left[u^{h\left(\sigma^1\right)}\left(t,x; W^{0},\mathcal{G}^0\right)\,\Big|\,W^{0},\,\sigma_{t}^{1}=y,\,B^{0},\,\mathcal{G}^0\right]dydt\Bigg] \\
& & \qquad \leq \mathbb{E}\Bigg[\int_{\mathbb{R}^{+}}M(t)\int_{\mathbb{R}^{+}}p_{t}\left(y\,|\,B^{0},\,\mathcal{G}^0\right) \nonumber \\
& & \qquad \qquad \qquad \qquad \times \mathbb{E}\left[\left(u^{h\left(\sigma^1\right)}\left(t,x; W^{0},\mathcal{G}^0\right)\right)^2\,\Big|\,W^{0},\,\sigma_{t}^{1}=y,\,B^{0},\,\mathcal{G}^0\right]dydt\Bigg] \\
& & \qquad = \mathbb{E}\left[\int_{\mathbb{R}^{+}}M(t)\mathbb{E}\left[\left(u^{h\left(\sigma^1\right)}\left(t,x; W^{0},\mathcal{G}^0\right)\right)^2\,\Big|\,W^{0},\,B^{0},\,\mathcal{G}^0\right]dt\right] \\
& & \qquad = \int_{\mathbb{R}^{+}}M(t)\mathbb{E}\left[\left(u^{h\left(\sigma^1\right)}\left(t,x; W^{0},\mathcal{G}^0\right)\right)^2\right]dt,
\end{eqnarray*}
since we can use the maximum principle given in Lemma~E2.2 from \cite{ERR}, the integrability of $M(\cdot)$ and the Dominated Convergence Theorem to show that the RHS of the last line tends to zero as $x \to 0^{+}$. The proof of the Theorem is now complete.
\end{proof}

\begin{rem}\label{rmkintegr}
Comparing the above proof with that of Theorem E1.2 in \cite{ERR} (in the case where $C_1$ is deterministic), we see that obtaining $H_0$ regularity when $W^0$ and $B^0$ are correlated is possible here because the bound $M(t)$ is deterministic, something we didn't have for the bounds obtained in the CIR volatility case. For the same reason, it is the randomness of $\displaystyle{M_{B^{0},\, \mathcal{G}^0}^{\alpha}(\cdot) = \sup_{y \in \mathbb{R}}\left(|y|^{\alpha}p_{\cdot}(y\,|\,B^{0},\,\mathcal{G}^0)\right)}$ which does not allow for $|y|^{\alpha}$-weighted integrability to be established for the derivative $u_x$ for $\alpha > 0$, even though we can have this for $u$ itself by using that $\left\Vert u(t,\,\cdot)\right\Vert_{L^{2}(\mathbb{R}^{+})}^{2} \leq \left\Vert u_{0}\right\Vert_{L^{2}(\mathbb{R}^{+})}^{2}$ for almost all $\left(t, \omega^0\right) \in \left[0, \, T\right] \times \Omega^0$ by Theorem~\ref{thm:4.1}. We will see that the regularity and the weighted integrability obtained here are exactly what we need for proving the results of the next two sections.  
\end{rem}

Now that we have a regular density $u(t,\,x,\,y) = u\left(t,\,x,\,y;\,W^{0},\,B^{0},\,\mathcal{G}^0\right)$ for the measure-valued process $v_{t}$, we can substitute $\int_{\mathbb{R}^{2}}f\cdot dv_{t}=\int_{\mathbb{R}^{2}}f(x,y)u(t,\,x,\,y)dxdy$
in the distributional SPDE of Theorem~\ref{thm:2.6} and integrate by parts to obtain
\begin{eqnarray}
u(t,\,x,\,y)&=& U_{0}(x, \, y \,|\,\mathcal{G}^0)-r\int_{0}^{t}\left(u(s,\,x,\,y)\right)_{x}ds \nonumber \\
& & \qquad +\frac{1}{2}\int_{0}^{t}h^{2}(y)\left(u(s,\,x,\,y)\right)_{x}ds-k\theta\int_{0}^{t}
\left(u(s,\,x,\,y)\right)_{y}ds \nonumber \\
& & \qquad +k\int_{0}^{t}\left(yu(s,\,x,\,y)\right)_{y}ds+\frac{1}{2}\int_{0}^{t}h^{2}(y)\left(u(s,\,x,\,y)\right)_{xx}ds
\nonumber \\
& & \qquad +\xi_{1}\rho_{3}\rho_{1}\rho_{2}\int_{0}^{t}\left(h\left(y\right)q(y)u(s,\,x,\,y)\right)_{xy}ds \nonumber \\
& & \qquad +\frac{\xi^{2}}{2}\int_{0}^{t}\left(q^2(y)u(s,\,x,\,y)\right)_{yy}ds \nonumber \\ 
& & \qquad -\rho_{1}\int_{0}^{t}h(y)\left(u(s,\,x,\,y)\right)_{x}dW_{s}^{0}\nonumber \\
& & \qquad -\xi\rho_{2}\int_{0}^{t}\left(q(y)u(s,\,x,\,y)\right)_{y}dB_{s}^{0}, \label{eq:4.12}
\end{eqnarray}
where $U_0(x, \, y \,|\,\mathcal{G}^0)$ stands for the initial density. This is the SPDE satisfied by the two-dimensional density $u(t,\,x,\,y)$, in which the second derivative in $x$ and the derivatives in $y$ are considered in a distributional sense (over the test space $C_{0}^{test}$ defined in Theorem~\ref{thm:2.6}, with an extra vanishing condition for $y \to \pm \infty$ to ensure that integration by parts with respect to $y$ does not leave boundary terms). 

\section{Higher regularity of initial-boundary value problem solutions}

In this section we establish even higher regularity for solutions to the stochastic initial-boundary value
problem satisfied by the density $u(t,\,x,\,y)$, which is needed for proving the crucial uniqueness result later. The proof of this result is interesting as it shows that the $|y|^{\alpha}$-weighted integrability obtained in the previous section for $u$ but not for its derivative $u_x$ (see Remark~\ref{rmkintegr}) is precisely what we need to cope with the unbounded coefficient $k(\theta - z)$, where the reason for this is that this problematic coefficient does not go with $u_x$ anywhere in the SPDE. All the lemmas presented here, which are needed for proving the main result of this section, are proved in the appendix, as they are minor modifications of lemmas used in \cite{HK17,ERR}. To simplify the notation of the previous section, we set $\tilde{L}_{\alpha}^2 = L_{|y|^{\alpha}}^{2}\left(\mathbb{R}^{+}
\times\mathbb{R}\right)$ and $\tilde{L}^2_{0,w} = L_{w^2(x)}^{2}\left(\mathbb{R}^{+}
\times\mathbb{R}\right)$. 

Before stating our result, we need to define our initial-boundary value problem explicitly. We give the following definition of an $\alpha$-solution to our problem for $\alpha\geq0$, the properties of which are all satisfied by the density function $u$ for all positive values of $\alpha$, as we have shown in the previous section.

\begin{defn}\label{problemdef}
For an  $\rho \in \mathbb{R}$, let $q: \mathbb{R} \mapsto \mathbb{R}^{+}$ and $h: \mathbb{R} \mapsto \mathbb{R}^{+}$ be continuous functions, and $U_0$ be a random function such that $U_{0}\in L^{2}\left(\Omega^0;\,\tilde{L}_{\alpha}^{2}
\right)$ and $\left(U_{0}\right)_{x}\in L^{2}\left(\Omega^0;\,\tilde{L}_{0, w}^{2}
\right)$ for some $\alpha > 0$.
Given $\rho$, $\alpha$ and the functions $q$, $h$ and $U_{0}$, we say that $u$ is an $\alpha$-solution to our problem when the following are satisfied;

\begin{enumerate}
\item   $u$ is adapted to the filtration $\{\sigma\left(\mathcal{G}^0,\,W_{t}^{0},\,B_{t}^{0}\right):\,t\geq0\}$
and belongs to the space $L_{\alpha}\cap H_{0}$, where $L_{\alpha}$
and $H_{0}$ are defined in Section~4.
\item  $u$ satisfies the SPDE
\begin{eqnarray}\label{finalspde}
u(t,x,y)&=& U_{0}(x,\,y)-r\int_{0}^{t}\left(u(s,x,y)\right)_{x}ds \nonumber \\
&& \qquad +\frac{1}{2}\int_{0}^{t}h^{2}(y)\left(u(s,x,y)\right)_{x}ds-k\theta\int_{0}^{t}\left(u(s,x,y)\right)_{y}ds 
\nonumber \\
&& \qquad +k\int_{0}^{t}\left(yu(s,\,x,\,y)\right)_{y}ds+\frac{1}{2}\int_{0}^{t}h^{2}(y)\left(u(s,x,y)\right)_{xx}ds
\nonumber \\
& & \qquad +\rho\int_{0}^{t}\left(h\left(y\right)q(y)u(s,\,x,\,y)\right)_{xy}ds \nonumber \\
& & \qquad +\frac{\xi^{2}}{2}\int_{0}^{t}\left(q^2(y)u(s,x,y)\right)_{yy}ds-\rho_{1}\int_{0}^{t}h(y)
\left(u(s,x,y)\right)_{x}dW_{s}^{0} \nonumber \\
& & \qquad -\xi\rho_{2}\int_{0}^{t}\left(q(y)u(s,x,y)\right)_{y}dB_{s}^{0}, \label{eq:5.1}
\end{eqnarray}
for all $x\geq0$ and $y \in \mathbb{R}$, where $u_{y}$, $u_{yy}$
and $u_{xx}$ are considered in the distributional sense over the space of test functions
\begin{eqnarray}
\tilde{C}_{0}^{test}=\{g\in C_{b}^{2}(\mathbb{R}^{+}\times\mathbb{R})&:&\,g(0,y) = 0 \;\forall y \in\mathbb{R}, \nonumber \\
&&\,\lim_{z \to \pm \infty}g(x,z) = \lim_{z \to \pm \infty}g_y(x,z) = 0 \;\forall x \in \mathbb{R}^{+}\}. \nonumber
\end{eqnarray}
\end{enumerate}
\end{defn}

\begin{rem}
For $\rho = \xi\rho_{3}\rho_{1}\rho_{2}$, where $\rho_3$ denotes the correlation coefficient between $W_{t}^0$ and $B_{t}^0$, we obtain the SPDE \eqref{eq:4.12}.
\end{rem}

We proceed now to our improved regularity result.

\begin{thm}\label{thm:5.2}
Fix the real number $\rho$ and the initial data function $U_{0}$. Let $u$ be an $\alpha$-solution to our problem for all $\alpha \geq 0$, where $q$ satisfies the conditions of Assumption~\ref{ass1} and $h$ is as in Theorem~\ref{thm:4.2}. Then, the weak derivative $u_{y}$ of $u$ exists and we have
    \[ 
    u_{y}\in L^{2}\left(\left[0,\,T\right]\times\Omega^0;\,\tilde{L}^2_{0, w}\right)
    \]
\end{thm}

\begin{rem}\label{2ndorderrem}
This result shows that any $u$ satisfying the conditions of Definition \ref{problemdef} for all $\alpha \geq 0$ is also weakly differentiable in $y$, and $u_y$ has good integrability properties just like $u_x$. It is possible that higher order derivatives belonging to certain weighted $L^2$ spaces exist as well, but this is something we will not investigate in this paper 
\end{rem}

The proof of Theorem~\ref{thm:5.2} is almost identical to the proof of \cite[Theorem 5.2]{HK17}, and we only need to adapt each step to the new properties of the vol-of-vol function $q$ which is not the square root function here. The idea is to test the SPDE against a heat kernel composed with an appropriate function, and derive an estimate involving some weighted $L^2$ norms of smooth approximations of $u$ and its derivatives. As we did in \cite{HK17}, we compose the smoothing heat kernel with a function which maps our volatility process to a constant volatility diffusion. This leads to the elimination of some exploding terms as $\epsilon \to 0^{+}$, allowing us to establish finiteness of certain weighted Sobolev norms. Therefore, we test our SPDE against
\[
\phi_{\epsilon}(z,\,y)=\frac{1}{\sqrt{2\pi\epsilon}}e^{-\frac{(Q(z)-y)^{2}}{2\epsilon}},
\;y,\,z\in\mathbb{R},
\]
where $Q(z) = \int_{z_0}^{z}\frac{1}{q(z')}dz'$ for some $z_0 \in \mathbb{R}$, as $\phi_{\epsilon} \in \tilde{C}_{0}^{test}$. 

Of course, we need to show that our kernel possesses nice smoothing and convergence (as $\epsilon \to 0^{+}$) properties. This is given in the following modifications of \cite[Lemma 5.3]{HK17} and \cite[Lemma 5.4]{HK17}, the proofs of which are significantly easier and are discussed in the Appendix (subsections~\ref{pf5.5} and \ref{pf5.6}).

\begin{lem}\label{lem:5.3}
Let $\left(\Lambda, \, \mu \right)$ be a measure space. For any function $u$ supported in $\Lambda \times \mathbb{R}$ we define
the functions
\[
J_{u,\epsilon}(\lambda,\,y)=\int_{\mathbb{R}}u(\lambda,\,z)\phi_{\epsilon}(z,\,y)dz
\]
and
\[
J_{u}(\lambda, \,y) = q\left(Q^{-1}(y)\right)u(\lambda,\,Q^{-1}(y)) \qquad y \in \mathbb{R} 
\]
Suppose now that $J_{u} \in L^{2}\left(\left(\Lambda, \, \mu \right);\,L^{2}\left(\mathbb{R}
\right)\right)$. Then we have the following regularity and convergence results;
\begin{enumerate}
\item $J_{u,\epsilon}(\cdot,\,\cdot)$ is smooth and for all $n \in \mathbb{N}$ we have
\[\frac{\partial^{n}}{\partial y^{n}}J_{u,\epsilon}(\cdot,\,\cdot) \in L^{2}\left(\left(\Lambda, \, \mu \right);\,L^{2}\left(\mathbb{R}
\right)\right);
\]
\item $J_{u,\epsilon}(\cdot,\,\cdot)\to J_{u}(\cdot,\,\cdot)$
strongly in $L^{2}\left(\left(\Lambda, \, \mu \right);\,L^{2}\left(\mathbb{R}
\right)\right)$, as $\epsilon\to0^{+}$.
\end{enumerate}.
\end{lem}

\begin{lem}\label{lem:5.4}
In the notation of lemma~\ref{lem:5.3}, assume that there exists a constant $C>0$ and an $n\in\mathbb{N}$ such that for any $\epsilon>0$ we have
\begin{equation}\label{eq:5.5}
\left\Vert \frac{\partial^{l}}{\partial y^{l}}J_{u,\epsilon}\left(\cdot,\,\cdot\right)\right\Vert _{L^{2}\left(\left(\Lambda, \, \mu \right);\,L^{2}\left(\mathbb{R}
\right)\right)}^{2} \leq C
\end{equation}
for some function $u$ supported in $\Lambda \times \mathbb{R}$ and all $l\in\{1,\,2,\,\ldots,\,n\}$. Then we have that $\frac{\partial^{l}}{\partial y^{l}}J_{u}\in L^{2}\left(\left(\Lambda, \, \mu \right);\,L^{2}\left(\mathbb{R}
\right)\right)$ and also $\frac{\partial^{l}}{\partial y^{l}}J_{u,\epsilon}\to\frac{\partial^{l}}{\partial y^{l}}J_{u}$
strongly in $L^{2}\left(\left(\Lambda, \, \mu \right);\,L^{2}\left(\mathbb{R}
\right)\right)$
as $\epsilon\to0^{+}$, for all $l\in\{1,\,2,\,\ldots,\,n\}$.
\end{lem}

Writing $\mu_{X}$ and $\mu_{X}^w$ for the standard Lebesgue measure on $X \subset \mathbb{R}$ and the weighted Lebesgue measure on $X \subset \mathbb{R}$ with weight $w^2(x)$ respectively, we can write 
\begin{equation}
L^{2}\left(\Omega^0;\,\tilde{L}_{0}^{2}\right) = L^{2}\left(\left(\Omega^0 \times \mathbb{R}^{+}, \, \mathbb{P}^{0} \times \mu_{\mathbb{R}^{+}}\right);\,L^{2}\left(\mathbb{R}
\right)\right), \nonumber  
\end{equation}
\begin{equation}
L^{2}\left(\Omega^0;\,\tilde{L}_{0, w}^{2}\right) = L^{2}\left(\left(\Omega^0 \times \mathbb{R}^{+}, \, \mathbb{P}^{0} \times \mu_{\mathbb{R}^{+}}^{w}\right);\,L^{2}\left(\mathbb{R}
\right)\right), \nonumber
\end{equation}
\begin{eqnarray}
L^{2}\left(\left[0, \, t \right] \times \Omega^0;\,\tilde{L}_{0}^{2}\right) = L^{2}\left(\left(\left[0, \, t \right] \times \Omega^0 \times \mathbb{R}^{+}, \, \mu_{\left[0, \, t \right]} \times \mathbb{P}^{0} \times \mu_{\mathbb{R}^{+}}\right);\,L^{2}\left(\mathbb{R}
\right)\right), \nonumber
\end{eqnarray}
and
\begin{equation}
L^{2}\left(\left[0, \, t \right] \times \Omega^0;\,\tilde{L}_{0, w}^{2}\right) = L^{2}\left(\left(\left[0, \, t \right] \times \Omega^0 \times \mathbb{R}^{+}, \, \mu_{\left[0, \, t \right]} \times \mathbb{P}^{0} \times \mu_{\mathbb{R}^{+}}^{w}\right);\,L^{2}\left(\mathbb{R}
\right)\right), \nonumber
\end{equation}
which are the spaces on which the above two lemmas will be used.

Now for an $\alpha$-solution $u$ to our problem for all $\alpha \geq 0$ we set:
\[
I_{\epsilon,g(z)}(s,\,x,\,y)=\int_{\mathbb{R}}g(z)u(s,\,x,\,z)\phi_{\epsilon}(z,\,y)dz
\]
for any function $g$ of $z$ and $\epsilon > 0$, and we have the following result

\begin{lem}\label{mainid}
For the $\alpha$-solution $u$ we have the following identity
\begin{eqnarray}\label{mainideq}
& & \left\Vert I_{\epsilon,1}(t,\,\cdot)\right\Vert _{L^{2}\left(\Omega^0;\,\tilde{L}_{0, w}^2\right)}^{2}  \nonumber \\ & & \qquad = \left\Vert \int_{\mathbb{R}}U_{0}(\cdot,\,z)\phi_{\epsilon}(z,\,\cdot)dz\right\Vert _{L^{2}\left(\Omega^0;\,\tilde{L}_{0, w}^2\right)}^{2} \nonumber \\
& & \qquad \qquad\qquad +r\int_{0}^{t}\left\Vert \mathbb{I}_{\left[0, \, 1\right] \times \mathbb{R}}(\cdot)I_{\epsilon,1}(s,\cdot)\right\Vert _{L^2\left(\Omega^0 ; \, \tilde{L}_{0}^2\right)}^2ds \nonumber \\
& & \qquad \qquad \qquad+\int_{0}^{t}\left\langle \frac{\partial}{\partial x}I_{\epsilon,h^{2}(z)}(s,\,\cdot),\,I_{\epsilon,1}(s,\,\cdot)\right\rangle _{L^{2}\left(\Omega^0;\,\tilde{L}_{0, w}^2\right)}ds \nonumber \\
& & \qquad \qquad\qquad+2k\theta\int_{0}^{t}\left\langle I_{\epsilon, Q'(z)}(s,\,\cdot),\,\frac{\partial}{\partial y}I_{\epsilon,1}(s,\,\cdot)\right\rangle _{L^{2}\left(\Omega^0;\,\tilde{L}_{0, w}^2\right)}ds \nonumber \\
& & \qquad\qquad\qquad -\xi^{2}\int_{0}^{t}\left\langle I_{\epsilon,q'(z)}(s,\,\cdot),\,\frac{\partial}{\partial y}I_{\epsilon,1}(s,\,\cdot)\right\rangle _{L^{2}\left(\Omega^0;\,\tilde{L}_{0, w}^2\right)}ds \nonumber \\
& & \qquad\qquad\qquad -2k\int_{0}^{t}\left\langle I_{\epsilon,zQ'(z)}(s,\,\cdot),\,\frac{\partial}{\partial y}I_{\epsilon,1}(s,\,\cdot)\right\rangle _{L^{2}\left(\Omega^0;\,\tilde{L}_{0, w}^2\right)}ds \nonumber \\
& & \qquad\qquad\qquad -\int_{0}^{t}\left\langle \frac{\partial}{\partial x}I_{\epsilon,h^{2}(z)}(s,\,\cdot),\,\frac{\partial}{\partial x}I_{\epsilon,1}(s,\,\cdot)\right\rangle _{L^{2}\left(\Omega^0;\,\tilde{L}_{0, w}^2\right)}ds \nonumber \\
& & \qquad\qquad\qquad -\int_{0}^{t}\left\langle \frac{\partial}{\partial x}I_{\epsilon,h^{2}(z)}(s,\,\cdot),\,\mathbb{I}_{\left[0, \, 1\right] \times \mathbb{R}}(\cdot)I_{\epsilon,1}(s,\,\cdot)\right\rangle _{L^{2}\left(\Omega^0;\,\tilde{L}_{0}^2\right)}ds \nonumber \\
& & \qquad \qquad\qquad+\rho_{1}^{2}\int_{0}^{t}\left\Vert \frac{\partial}{\partial x}I_{\epsilon,h(z)}(s,\,\cdot)\right\Vert _{L^{2}\left(\Omega^0;\,\tilde{L}_{0, w}^2\right)}^{2}ds
\nonumber \\
& & \qquad\qquad\qquad -\xi^{2}\left(1-\rho_{2}^{2}\right)\int_{0}^{t}\left\Vert \frac{\partial}{\partial y}I_{\epsilon,1}(s,\,\cdot)\right\Vert _{L^{2}\left(\Omega^0;\,\tilde{L}_{0, w}^2\right)}^{2}ds \nonumber \\
& & \qquad\qquad\qquad - 2\left(\rho - \xi\rho_{3}\rho_{1}\rho_{2}\right) \nonumber \\
& & \qquad \qquad \qquad\qquad\qquad\times\int_{0}^{t}\left\langle \frac{\partial}{\partial x}I_{\epsilon,h\left(z\right)}(s,\cdot),\,\frac{\partial}{\partial y}I_{\epsilon,1}(s,\cdot)\right\rangle _{L^2\left(\Omega^0; \, \tilde{L}_{0, w}^2\right)}ds. 
\label{eq:5.17}
\end{eqnarray}
All the terms in the above identity are finite.
\end{lem}

The proof of the above lemma can also be found in the appendix (subsection~\ref{pf5.7}). We are now ready to prove our main result

\begin{proof}[\textbf{Proof of Theorem 5.3}]
For all the inner products in identity \eqref{mainideq} given in Lemma~\ref{mainid}, except the first and the tenth, we can use the Cauchy-Schwartz inequality and then Young's inequality, $ab\leq\frac{a^{2}}{4C}+Cb^{2}$, for the products of norms in order to obtain
\begin{eqnarray}
&&\left\Vert I_{\epsilon,1}(t,\,\cdot)\right\Vert _{L^{2}\left(\Omega^0;\,\tilde{L}_{0,w}^2\right)}^{2} \nonumber \\
&& \qquad \leq \left\Vert \int_{\mathbb{R}}U_{0}(\cdot,\,z)\phi_{\epsilon}(z,\,\cdot)dz\right\Vert _{L^{2}\left(\Omega^0;\,\tilde{L}_{0,w}^2\right)}^{2} +r\int_{0}^{t}\left\Vert \mathbb{I}_{\left[0, \, 1\right] \times \mathbb{R}}(\cdot)I_{\epsilon,1}(s,\cdot)\right\Vert _{L^2\left(\Omega^0 ; \, \tilde{L}_{0}^2\right)}^2ds \nonumber \\
&& \qquad \qquad \qquad+\int_{0}^{t}\left\langle \frac{\partial}{\partial x}I_{\epsilon,h^{2}(z)}(s,\,\cdot),\,I_{\epsilon,1}(s,\,\cdot)\right\rangle _{L^{2}\left(\Omega^0;\,\tilde{L}_{0,w}^2\right)}ds \nonumber \\
&& \qquad \qquad \qquad +k\theta\int_{0}^{t}C\left\Vert I_{\epsilon,Q'(z)}(s,\,\cdot)\right\Vert _{L^{2}\left(\Omega^0;\,\tilde{L}_{0,w}^2\right)}^{2}ds \nonumber \\
&& \qquad\qquad\qquad+k\theta\int_{0}^{t}\frac{1}{C}\left\Vert \frac{\partial}{\partial y}I_{\epsilon,1}(s,\,\cdot)\right\Vert _{L^{2}\left(\Omega^0;\,\tilde{L}_{0,w}^2\right)}^{2}ds 
\nonumber \\
&& \qquad \qquad \qquad+\xi^{2}\int_{0}^{t}C\left\Vert I_{\epsilon, q'(z)}(s,\,\cdot)\right\Vert _{L^{2}\left(\Omega^0;\,\tilde{L}_{0,w}^2\right)}^{2}ds \nonumber \\
&& \qquad\qquad\qquad +\xi^{2}\int_{0}^{t}\frac{1}{4C}\left\Vert \frac{\partial}{\partial y}I_{\epsilon,1}(s,\,\cdot)\right\Vert _{L^{2}\left(\Omega^0;\,\tilde{L}_{0,w}^2\right)}^{2}ds \nonumber \\
&& \qquad \qquad \qquad+k\int_{0}^{t}C\left\Vert I_{\epsilon,zQ'(z)}(s,\,\cdot)\right\Vert _{L^{2}\left(\Omega^0;\,\tilde{L}_{0,w}^2\right)}^{2}ds \nonumber \\
&& \qquad\qquad\qquad +k\int_{0}^{t}\frac{1}{C}\left\Vert \frac{\partial}{\partial y}I_{\epsilon,1}(s,\,\cdot)\right\Vert _{L^{2}\left(\Omega^0;\,\tilde{L}_{0,w}^2\right)}^{2}ds
\nonumber \\
&& \qquad \qquad \qquad-\int_{0}^{t}\left\langle \frac{\partial}{\partial x}I_{\epsilon,h^{2}(z)}(s,\,\cdot),\,\frac{\partial}{\partial x}I_{\epsilon,1}(s,\,\cdot)\right\rangle _{L^{2}\left(\Omega^0;\,\tilde{L}_{0,w}^2\right)}ds \nonumber \\
&& \qquad\qquad\qquad +\rho_{1}^{2}\int_{0}^{t}\left\Vert \frac{\partial}{\partial x}I_{\epsilon,h(z)}(s,\,\cdot)\right\Vert _{L^{2}\left(\Omega^0;\,\tilde{L}_{0,w}^2\right)}^{2}ds \nonumber \\
&& \qquad \qquad \qquad-\int_{0}^{t}\left\langle \frac{\partial}{\partial x}I_{\epsilon,h^{2}(z)}(s,\,\cdot),\,\mathbb{I}_{\left[0, \, 1\right] \times \mathbb{R}}(\cdot)I_{\epsilon,1}(s,\,\cdot)\right\rangle _{L^{2}\left(\Omega^0;\,\tilde{L}_{0}^2\right)}ds \nonumber \\
&& \qquad\qquad\qquad -\xi^{2}\left(1-\rho_{2}^{2}\right)\int_{0}^{t}\left\Vert \frac{\partial}{\partial y}I_{\epsilon,1}(s,\,\cdot)\right\Vert _{L^{2}\left(\Omega^0;\,\tilde{L}_{0,w}^2\right)}^{2}ds \nonumber \\
&& \qquad \qquad \qquad
+ C\left|\rho - \xi\rho_{3}\rho_{1}\rho_{2}\right| \int_{0}^{t}\left\Vert \frac{\partial}{\partial x}I_{\epsilon,h(z)}(s,\,\cdot)\right\Vert _{L^{2}\left(\Omega^0;\,\tilde{L}_{0,w}^2\right)}^{2}ds \nonumber \\
&& \qquad\qquad\qquad + \frac{1}{C}\left|\rho - \xi\rho_{3}\rho_{1}\rho_{2}\right| \int_{0}^{t}\left\Vert \frac{\partial}{\partial y}I_{\epsilon,1}(s,\,\cdot)\right\Vert _{L^{2}\left(\Omega^0;\,\tilde{L}_{0,w}^2\right)}^{2}ds 
\label{eq:5.18}
\end{eqnarray}
and this for any $C > 0$. Taking a large enough $C$ such that
\begin{eqnarray*}
\left(k\theta + \frac{\xi^{2}}{4}+\left|\rho - \xi\rho_{3}\rho_{1}\rho_{2}\right|+k\right)\frac{1}{C}<\xi^{2}\left(1-\rho_{2}^{2}\right),
\end{eqnarray*}
from \eqref{eq:5.18} we can derive an estimate of the form
\begin{eqnarray}
& & \left\Vert I_{\epsilon,1}(t,\,\cdot)\right\Vert _{L^{2}\left(\Omega^0;\,\tilde{L}_{0,w}^2\right)}^{2}+M_1\int_{0}^{t}\left\Vert \frac{\partial}{\partial y}I_{\epsilon,1}(s,\,\cdot)\right\Vert _{L^{2}\left(\Omega^0;\,\tilde{L}_{0,w}^2\right)}^{2}ds \nonumber \\
& & \qquad \leq\left\Vert \int_{\mathbb{R}}U_{0}(\cdot,\,z)\phi_{\epsilon}(z,\,\cdot)dz\right\Vert _{L^{2}\left(\Omega^0;\,\tilde{L}_{0,w}^2\right)}^{2} + M_2\int_{0}^{t}\left\Vert \frac{\partial}{\partial x}I_{\epsilon,h(z)}(s,\,\cdot)\right\Vert _{L^{2}\left(\Omega^0;\,\tilde{L}_{0,w}^2\right)}^{2}ds \nonumber \\
& & \qquad \qquad +M_2\sum_{g(z)\in\left\{ Q'(z), zQ'(z), q'(z)\right\} }\int_{0}^{t}\left\Vert I_{\epsilon,g(z)}(s,\,\cdot)\right\Vert _{L^{2}\left(\Omega^0;\,\tilde{L}_{0,w}^2\right)}^{2}ds \nonumber \\
& & \qquad \qquad +r\int_{0}^{t}\left\Vert \mathbb{I}_{\left[0, \, 1\right] \times \mathbb{R}}(\cdot)I_{\epsilon,1}(s,\cdot)\right\Vert _{L^{2}\left(\Omega^0;\,\tilde{L}_{0}^2\right)}^2ds \nonumber \\
&& \qquad \qquad -\int_{0}^{t}\left\langle \frac{\partial}{\partial x}I_{\epsilon,h^{2}(z)}(s,\,\cdot),\,\mathbb{I}_{\left[0, \, 1\right] \times \mathbb{R}}(\cdot)I_{\epsilon,1}(s,\,\cdot)\right\rangle _{L^{2}\left(\Omega^0;\,\tilde{L}_{0}^2\right)}ds \nonumber \\
& & \qquad \qquad +\int_{0}^{t}\left\langle \frac{\partial}{\partial x}I_{\epsilon,h^{2}(z)}(s,\,\cdot),\,I_{\epsilon,1}(s,\,\cdot)\right\rangle _{L^{2}\left(\Omega^0;\,\tilde{L}_{0,w}^2\right)}ds \nonumber \\
& & \qquad \qquad -\int_{0}^{t}\left\langle \frac{\partial}{\partial x}I_{\epsilon,h^{2}(z)}(s,\,\cdot),\,\frac{\partial}{\partial x}I_{\epsilon,1}(s,\,\cdot)\right\rangle _{L^{2}\left(\Omega^0;\,\tilde{L}_{0,w}^2\right)}ds, \label{eq:5.19}
\end{eqnarray}
for some positive constants $M_1$ and $M_2$. Then, with the notation of Lemma~\ref{lem:5.3} we have $I_{\epsilon,g(z)}=J_{g\cdot u,\epsilon}$ and $\frac{\partial}{\partial x}I_{\epsilon,g(z)}=J_{g\cdot u_{x},\epsilon}$ for any function $g$, and by using that lemma, the above two quantities can be shown to converge in $L^{2}\left(\left[0, \, t \right] \times \Omega^0;\,\tilde{L}_{0,w}^{2}\right)$ to
\[
J_{g\cdot u}(s,\,x,\,v)=g\left(Q^{-1}(v)\right)q\left(Q^{-1}(v)\right)u\left(s,\,x,\,Q^{-1}(v)\right)
\]
and 
\[
J_{g\cdot u_{x}}(s,\,x,\,v)=g\left(Q^{-1}(v)\right)q\left(Q^{-1}(v)\right)u_{x}\left(s,\,x,\,Q^{-1}(v)\right)
\]
respectively as $\epsilon \to 0^{+}$, provided that the last two quantities belongs to the space $L^{2}\left(\left[0, \, t \right] \times \Omega^0;\,\tilde{L}_{0,w}^{2}\right)$. To verify this 
we use the transformation $v \mapsto Q(v)$ (with the range of $Q$ being clearly $\mathbb{R}$) to compute
\begin{eqnarray}
&&\left\Vert J_{g\cdot u}\right\Vert _{L_{w^2(x)}^{2}\left(\left[0,\,t\right]\times\Omega^0\times\mathbb{R}^{+}\times \mathbb{R}\right)}^{2} \nonumber \\
&& \qquad = \int_{0}^{t}\mathbb{E}\left[\int_{\mathbb{R}^+}\int_{\mathbb{R}}w^2(x)g^2\left(y\right)q\left(y\right)u^2\left(s,\,x,\,y\right)dvdx\right]ds \label{eq:5.20}
\end{eqnarray}
and
\begin{eqnarray}
&&\left\Vert J_{g\cdot u_x}\right\Vert _{L_{w^2(x)}^{2}\left(\left[0,\,t\right]\times\Omega^0\times\mathbb{R}^{+}\times \mathbb{R}\right)}^{2} \nonumber \\
&& \qquad = \int_{0}^{t}\mathbb{E}\left[\int_{\mathbb{R}^+}\int_{\mathbb{R}}w^2(x)g^2\left(y\right)q\left(y\right)u_{x}^2\left(s,\,x,\,y\right)dvdx\right]ds \label{eq:5.21}
\end{eqnarray}
where \eqref{eq:5.20} is always finite as all functions $g$ in \eqref{eq:5.19} which are paired with non-derivative terms have polynomial growth, and the same holds for \eqref{eq:5.21} as well since all $g$ in \eqref{eq:5.19} which are paired with derivative terms are bounded. 

Recalling now that the inner products are continuous, we deduce that all the terms in the RHS of \eqref{eq:5.19} are convergent as $\epsilon\to 0^{+}$, and thus the RHS of \eqref{eq:5.19} is bounded in $\epsilon$. This allows us to apply Lemma~\ref{lem:5.4} to the $y$-derivative term in the LHS of \eqref{eq:5.19} and deduce that $\frac{\partial}{\partial v}J_{u} = \frac{\partial}{\partial v}q\left(Q^{-1}(v)\right)u\left(s,\,x,\,Q^{-1}(v)\right)$
exists in $L^{2}\left(\left[0, \, t \right] \times \Omega^0;\,\tilde{L}_{0,w}^{2}\right)$ and that, in that space, we have $\frac{\partial}{\partial v}I_{\epsilon,1} = \frac{\partial}{\partial v}J_{u,\epsilon} \to \frac{\partial}{\partial v}J_u$ as $\epsilon \to 0^{+}$. It follows then that the weak derivative $\frac{\partial}{\partial y}u\left(s,\,x,\,y\right)$
exists, and by using the product rule and the inequality $(a - b)^2 \leq 2a^2 + 2b^2$ we can obtain
\begin{eqnarray}
& & \int_{0}^{t}\mathbb{E}\left[\int_{\mathbb{R}^+}\int_{\mathbb{R}}u_{y}^{2}(s,\,x,\,y)dydx\right]ds \nonumber \\
&& \qquad=\int_{0}^{t}\mathbb{E}\left[\int_{\mathbb{R}^+}\int_{\mathbb{R}}q\left(Q^{-1}(v)\right)u_{y}^{2}\left(s,\,x,\,Q^{-1}(v)\right)dvdx\right]ds \nonumber \\
&& \qquad=\int_{0}^{t}\mathbb{E}\left[\int_{\mathbb{R}^+}\int_{\mathbb{R}}q^{-1}\left(Q^{-1}(v)\right)\left(\frac{\partial}{\partial v}\left(u\left(s,\,x,\,Q^{-1}(v)\right)\right)\right)^2dvdx\right]ds \nonumber \\
&&  \qquad \leq 2\int_{0}^{t}\int_{\mathbb{R}^{+}}\int_{\mathbb{R}} \mathbb{E}\left[q^{-3}\left(Q^{-1}(v)\right)\left(\frac{\partial}{\partial v}\left(q\left(Q^{-1}(v)\right)u\left(s,\,x,\,Q^{-1}(v)\right)\right)\right)^2\right]dvdxds \nonumber \\
&&  \qquad \qquad + 2\int_{0}^{t}\int_{\mathbb{R}^{+}}\int_{\mathbb{R}}\mathbb{E}\left[q^{-1}\left(Q^{-1}(v)\right)\left(q'\left(Q^{-1}(v)\right)u\left(s,\,x,\,Q^{-1}(v)\right)\right)^2\right]dvdxds. \nonumber 
\label{eq:5.22}
\end{eqnarray}
Since $\frac{\partial}{\partial v}J_{u} = \frac{\partial}{\partial v}q\left(Q^{-1}(v)\right)u\left(s,\,x,\,Q^{-1}(v)\right) \in L^{2}\left(\left[0, \, t \right] \times \Omega^0;\,\tilde{L}_{0,w}^{2}\right)$, the RHS of the last expression is finite, so we finally obtain $u_{y}\in L^{2}\left(\left[0,\,T\right]\times\Omega^0;\,\tilde{L}^2_{0,w}\right)$.
\end{proof}

\begin{rem}\label{rem:5.7} As in the CIR volatility case (see \cite[Remark~5.7]{HK17}), the flexibility in the choice of $\rho$ allows for an extension of our results to the case where the idiosyncratic noises have nonzero correlation.
\end{rem}

\section{Uniqueness of solutions}

In this section we establish uniqueness of solutions to our problem (Definition \ref{problemdef}) under the assumptions of Theorem~\ref{thm:5.2}. Thus, we assume again that the function $q$ in our SPDE \eqref{finalspde} satisfies the conditions of Assumption~\ref{ass1}. As we have mentioned in Remark~\ref{rem3}, these functions are positive smooth functions behaving almost like a positive constant near infinity. Moreover, we will need to assume that the coefficients of \eqref{finalspde} satisfy the condition
\begin{eqnarray}\label{corrcondition}
\left|\rho - \xi\rho_{3}\rho_{1}\rho_{2}\right| \leq \xi\sqrt{1 - \rho^{2}_{1}}\sqrt{1 - \rho^{2}_{2}},
\end{eqnarray}
which is always satisfied when the SPDE arises from a large portfolio model (even when the idiosyncratic noises are correlated, see Remark~\ref{rem:5.7}). 

Since our SPDE \eqref{finalspde} is linear, for any two distinct solutions $u_1$ and $u_2$ with the same initial data we have that $u = u_1 - u_2$ is a non-zero solution with zero initial data. Therefore, to establish pathwise uniqueness of solutions, we pick an arbitrary $u$ that is an $\alpha$-solution to our problem with initial data $U_0 = 0$ for all $\alpha \geq 0$, and it suffices to show that $u = 0$. 

The natural approach to is to derive an optimal estimate for the $L^2$-norm of $u$ by using the SPDE \eqref{finalspde}, but for the convenience of the reader, we will first sketch this technique for a much simpler 1-dimensional SPDE with a Dirichlet boundary condition:
\begin{eqnarray}\label{simple}
u(t, x) &=& -r\int_0^tu_{x}(s, x)ds + \int_0^tu_{xx}(s, x)ds + \rho\int_0^tu_{x}(s, x)dW^0_s \\
u(t, 0) &=& 0
\end{eqnarray}
for $r \in \mathbb{R}$ and $\rho \in (-1, 1)$. Under sufficient regularity, we can use Ito's formula on $u^2(t, x)$ which, under the above dynamics, yields:
\begin{eqnarray}\label{squared}
& &u^2(t, x) = -2r\int_0^tu(s, x)u_{x}(s, x)ds \nonumber \\
& & \qquad \qquad \qquad \qquad + 2\int_0^tu(s, x)u_{xx}(s, x)ds + \rho^2\int_0^t(u_{x}(s, x))^2ds + M_t^x 
\end{eqnarray}
where $M_t^x$ is a stochastic integral that depends on $x$. Then, we take expectations to eliminate the martingale term $M_t^x$, we integrate in $x$ over $(0, +\infty)$ which kills the term $2r u(s,x)u_x(s,x) = (r u^2(s,x))_x$ due to $u(t,0) = 0$, and we use integration by parts to write $\int_{0}^{+\infty}u(s, x)u_{xx}(s, x)dx = - \int_{0}^{+\infty}u_x^2(s, x)dx$. This leads to the identity:
\begin{eqnarray}\label{energy}
\int_0^{+\infty}\mathbb{E}[u^2(t, x)]dx + (1 - \rho^2)\int_0^t\int_0^{+\infty}\mathbb{E}[u_x^2(s, x)]dxds &=& 0
\end{eqnarray}
which is sufficient to deduce that $u = 0$. However, if the existence of the second derivative $u_{xx}$ is not guaranteed, we follow the above steps on a smooth approximation of \eqref{simple}, which leads to a smoothed version of \eqref{energy} and then we can obtain \eqref{energy} by taking a limit. Moreover, if we were forced by regularity constraints to work with the weighted $L^2$ norm $\Vert u \Vert_{L^2_{w^2(x)}} = \int_0^{+\infty}w^2(x)u^2(s, x)dx$, which is the case in our original problem due to difficulties that arise if we try to smooth \eqref{finalspde} not only in $y$ but also in $x$, we would need to integrate \eqref{squared} against the weight function $w^2(x) = \min\{x, 1\}$. Then integration by parts in the integrals $-r\int_0^{+\infty}w^2(x)( u^2(s, x))_xdx$ and $2\int_{0}^{+\infty}w^2(x)u(s, x)u_{xx}(s, x)dx$ would leave the two extra residual terms $r\int_0^1 u^2(s, x)dx$ and $-u^2(s, 1)$, giving the identity:
\begin{eqnarray}\label{energy2}
\int_0^{+\infty}\mathbb{E}[u^2(t, x)]dx &+& (1 - \rho^2)\int_0^t\int_0^{+\infty}\mathbb{E}[u_x^2(s, x)]dxds \nonumber \\
&-& r\int_0^t\int_0^{1}\mathbb{E}[u^2(s, x)]dxds + \int_0^t\mathbb{E}[u^2(s, 1)]ds = 0
\end{eqnarray}
instead of \eqref{energy}, which is sufficient to deduce that $u = 0$ only when $r \leq 0$. Thus, if $r > 0$, we need to work under a probability measure $\mathbb{Q}$ which is obtained by Girsanov's theorem, under which $W_t^0$ becomes a Brownian motion with arbitrary drift $\tilde{r}t$ for $t \in [0, T]$, and thus we get the same SPDE with $r$ replaced with $r + \tilde{r}$ which can be made negative by choice of $\tilde{r}$. The weighted $L^2$ norm of $u$ being zero implies that $\mathbb{Q}$ - almost surely we have $u(t, x) = 0$ for all $t \in [0, T]$ and $x \geq 0$, and since $\mathbb{Q}$ is equivalent to the original probability measure $\mathbb{P}$ we obtain pathwise uniqueness of solutions.

We now let $u$ be a solution to \eqref{finalspde} with $U_0 = 0$. To obtain the counterpart of \eqref{energy2} we take $\epsilon \to 0^{+}$ in a smoothed version of it, which is precisely the identity of Lemma~\ref{mainid} without the term on the second line, and we use the regularity established by Theorem~\ref{thm:5.2}, i.e the existence of $\frac{\partial}{\partial v}J_{u} = \frac{\partial}{\partial v}q\left(Q^{-1}(v)\right)u\left(s,\,x,\,Q^{-1}(v)\right)$ in $L^{2}\left(\left[0, \, t \right] \times \Omega^0;\,\tilde{L}_{0,w}^{2}\right)$ with the convergence $\frac{\partial}{\partial v}I_{\epsilon,1} = \frac{\partial}{\partial v}J_{u,\epsilon} \to \frac{\partial}{\partial v}J_u$ as $\epsilon \to 0^{+}$ in that space. To pass to the limit we also use Lemma~\ref{lem:5.3} and the continuity of $L^2$ inner products. Then, we can check that the counterpart of \eqref{energy2} can be written as:
\begin{eqnarray}\label{limid}
& & \int_{\mathbb{R}^{+}}\int_{\mathbb{R}} w^2(x)\mathbb{E}\left[q^2\left(Q^{-1}(v)\right)u^2\left(t,\,x,\,Q^{-1}(v)\right)\right]dxdv \nonumber \\
& & \qquad\qquad\qquad\qquad = N_1(t) + N_2(t) + N_3(t) + N_4(t) + N_5(t) + N_6(t), 
\end{eqnarray}
where: 
\begin{eqnarray}
&& N_1(t) = r\int_{0}^{t}\int_{0}^{1}\int_{\mathbb{R}}\mathbb{E}\left[q^2\left(Q^{-1}(v)\right)u^2\left(t,\,x,\,Q^{-1}(v)\right)\right]dxdvds,
\end{eqnarray}
is the $\epsilon \to 0$ limit of 
\begin{eqnarray*}
r\int_{0}^{t}\left\Vert \mathbb{I}_{\left[0, \, 1\right] \times \mathbb{R}}(\cdot)I_{\epsilon,1}(s,\cdot)\right\Vert _{L^2\left(\Omega^0 ; \, \tilde{L}_{0}^2\right)}^2ds
\end{eqnarray*}
from the identity of Lemma~\ref{mainid}, and $-N_1(t)$ is the counterpart of $r\int_0^t\int_0^{1}\mathbb{E}[u^2(s, x)]dxds$ from \eqref{energy2};
\begin{eqnarray}
&& N_2(t) = \int_{0}^{t} \int_{\mathbb{R}^{+}}\int_{\mathbb{R}} w^2(x)\mathbb{E}\left[\left(h^2\left(Q^{-1}(v)\right)q\left(Q^{-1}(v)\right)u\left(s,\,x,\,Q^{-1}(v)\right)\right)_{x}\right. \nonumber \\
& & \qquad \qquad \qquad \qquad \qquad \qquad \qquad \times \left.q\left(Q^{-1}(v)\right)u\left(s,\,x,\,Q^{-1}(v)\right)\right]dxdvds, \nonumber \\
\end{eqnarray}
and
\begin{eqnarray}\label{N3}
&& N_3(t) = 2k\theta\int_{0}^{t}\int_{\mathbb{R}^{+}}\int_{\mathbb{R}} w^2(x)\mathbb{E}\left[Q'\left(Q^{-1}(v)\right)q\left(Q^{-1}(v)\right)u\left(s,\,x,\,Q^{-1}(v)\right)\right. \nonumber \\
& & \qquad \qquad \quad \qquad \qquad \qquad \qquad \qquad \times\left.\left(q\left(Q^{-1}(v)\right)u\left(s,\,x,\,Q^{-1}(v)\right)\right)_{v}\right]dxdvds \nonumber \\
& & \qquad \qquad 
-2k\int_{0}^{t}\int_{\mathbb{R}^{+}}\int_{\mathbb{R}} w^2(x)\mathbb{E}\left[vQ'\left(Q^{-1}(v)\right)q\left(Q^{-1}(v)\right)u\left(s,\,x,\,Q^{-1}(v)\right)\right. \nonumber \\
& & \qquad \qquad \qquad \qquad \qquad \qquad \qquad \qquad  \times \left.\left(q\left(Q^{-1}(v)\right)u\left(s,\,x,\,Q^{-1}(v)\right)\right)_{v}\right]dxdvds, \nonumber \\
\end{eqnarray}
are the $\epsilon \to 0$ limits of
\begin{eqnarray*}
\int_{0}^{t}\left\langle \frac{\partial}{\partial x}I_{\epsilon,h^{2}(z)}(s,\,\cdot),\,I_{\epsilon,1}(s,\,\cdot)\right\rangle _{L^{2}\left(\Omega^0;\,\tilde{L}_{0, w}^2\right)}ds
\end{eqnarray*}
and
\begin{eqnarray*}
& & 2k\theta\int_{0}^{t}\left\langle I_{\epsilon, Q'(z)}(s,\,\cdot),\,\frac{\partial}{\partial y}I_{\epsilon,1}(s,\,\cdot)\right\rangle _{L^{2}\left(\Omega^0;\,\tilde{L}_{0, w}^2\right)}ds \nonumber \\
& & \qquad \qquad -2k\int_{0}^{t}\left\langle I_{\epsilon,zQ'(z)}(s,\,\cdot),\,\frac{\partial}{\partial y}I_{\epsilon,1}(s,\,\cdot)\right\rangle _{L^{2}\left(\Omega^0;\,\tilde{L}_{0, w}^2\right)}ds \nonumber \\
\end{eqnarray*}
respectively from the identity of Lemma~\ref{mainid}, which do not have a counterpart in \eqref{energy2} as they occur from extra first order derivative terms in the SPDE \eqref{finalspde}, and specifically terms like $r\int_0^tu_{x}(s, x)ds$ in \eqref{simple} but with $y$-dependent coefficients and differentiation in either $x$ or $y$;
\begin{eqnarray}
& & N_4(t) = -\int_{0}^{t}\int_{\mathbb{R}^{+}}\int_{\mathbb{R}} w^2(x)\mathbb{E}\left[\left(h^2\left(Q^{-1}(v)\right)q\left(Q^{-1}(v)\right)u\left(s,\,x,\,Q^{-1}(v)\right)\right)_{x}\right. \nonumber \\ 
& & \qquad \qquad \qquad \qquad \quad \qquad  \qquad \qquad \times \left.\left(q\left(Q^{-1}(v)\right)u\left(s,\,x,\,Q^{-1}(v)\right)\right)_{x}\right]dxdvds \nonumber \\
& &\qquad \qquad -\xi^{2}\int_{0}^{t}\int_{\mathbb{R}^{+}}\int_{\mathbb{R}} w^2(x) \mathbb{E}\left[\left(\left(q\left(Q^{-1}(v)\right)u\left(s,\,x,\,Q^{-1}(v)\right)\right)_{v}\right)^2\right]dxdvds \nonumber \\
& & \qquad \qquad -\xi^{2}\int_{0}^{t}\int_{\mathbb{R}^{+}}\int_{\mathbb{R}} w^2(x)\mathbb{E}\left[q'\left(Q^{-1}(v)\right)q\left(Q^{-1}(v)\right)u\left(s,\,x,\,Q^{-1}(v)\right)\right. \nonumber \\
& & \qquad \qquad \qquad \qquad \qquad\qquad\qquad\qquad \times \left.\left(q\left(Q^{-1}(v)\right)u\left(s,\,x,\,Q^{-1}(v)\right)\right)_{v}\right]dxdvds \nonumber \\
& & \qquad \qquad - 2\rho\int_{0}^{t}\int_{\mathbb{R}^{+}}\int_{\mathbb{R}} w^2(x) \mathbb{E}\left[\left(h\left(Q^{-1}(v)\right)q\left(Q^{-1}(v)\right)u\left(s,\,x,\,Q^{-1}(v)\right)\right)_{x}\right. \nonumber \\
& & \qquad \qquad\qquad \qquad \qquad\qquad \qquad \qquad \times \left.\left(q\left(Q^{-1}(v)\right)u\left(s,\,x,\,Q^{-1}(v)\right)\right)_{v}\right]dxdvds \nonumber \\
\end{eqnarray}
is the $\epsilon \to 0$ limit of 
\begin{eqnarray*}
& &-\int_{0}^{t}\left\langle \frac{\partial}{\partial x}I_{\epsilon,h^{2}(z)}(s,\,\cdot),\,\frac{\partial}{\partial x}I_{\epsilon,1}(s,\,\cdot)\right\rangle _{L^{2}\left(\Omega^0;\,\tilde{L}_{0, w}^2\right)}ds \nonumber \\
& & \qquad\qquad -\xi^{2}\int_{0}^{t}\left\Vert \frac{\partial}{\partial y}I_{\epsilon,1}(s,\,\cdot)\right\Vert _{L^{2}\left(\Omega^0;\,\tilde{L}_{0, w}^2\right)}^{2}ds \nonumber \\
& & \qquad\qquad -\xi^{2}\int_{0}^{t}\left\langle I_{\epsilon,q'(z)}(s,\,\cdot),\,\frac{\partial}{\partial y}I_{\epsilon,1}(s,\,\cdot)\right\rangle _{L^{2}\left(\Omega^0;\,\tilde{L}_{0, w}^2\right)}ds \nonumber \\
& & \qquad\qquad - 2\rho\int_{0}^{t}\left\langle \frac{\partial}{\partial x}I_{\epsilon,h\left(z\right)}(s,\cdot),\,\frac{\partial}{\partial y}I_{\epsilon,1}(s,\cdot)\right\rangle _{L^2\left(\Omega^0; \, \tilde{L}_{0, w}^2\right)}ds
\end{eqnarray*}
which is the part of the identity of Lemma~\ref{mainid} that occurs by applying integration by parts on terms that involve second order derivatives, so $-N_4(t)$ is the counterpart of $\int_0^t\int_0^{+\infty}\mathbb{E}[u_x^2(s, x)]dxds$ from \eqref{energy2}, and its third component contains only one derivative of first order as it arises from the presence of a $y$-dependent coefficient when integration by parts is applied;
\begin{eqnarray}
& & N_5(t) = \rho_{1}^{2}\int_{0}^{t}\int_{\mathbb{R}^{+}}\int_{\mathbb{R}} w^2(x) \mathbb{E}\big[\big(h^2\left(Q^{-1}(v)\right) \nonumber \\
& & \qquad\qquad\qquad\qquad \qquad \qquad \qquad \quad \times \big(q\left(Q^{-1}(v)\right)u\left(s,\,x,\,Q^{-1}(v)\right)\big)_{x}\big)^2\big]dxdvds \nonumber \\
& & \qquad \quad + \xi^{2}\rho_{2}^{2}\int_{0}^{t}\int_{\mathbb{R}^{+}}\int_{\mathbb{R}} w^2(x) \mathbb{E}\left[\left(\left(q\left(Q^{-1}(v)\right)u\left(s,\,x,\,Q^{-1}(v)\right)\right)_{v}\right)^2\right]dxdvds \nonumber \\
& & \quad \qquad + 2 \xi\rho_{3}\rho_{1}\rho_{2} \int_{0}^{t}\int_{\mathbb{R}^{+}}\int_{\mathbb{R}} w^2(x) \mathbb{E}\left[\left(h\left(Q^{-1}(v)\right)q\left(Q^{-1}(v)\right)u\left(s,\,x,\,Q^{-1}(v)\right)\right)_{x}\right. \nonumber \\
& & \quad \qquad\qquad \qquad\qquad \qquad\qquad \qquad \qquad\quad \times \left.\left(q\left(Q^{-1}(v)\right)u\left(s,\,x,\,Q^{-1}(v)\right)\right)_{v}\right]dxdvds \nonumber \\
\end{eqnarray}
is the $\epsilon \to 0$ limit of 
\begin{eqnarray*}
& &\rho_{1}^{2}\int_{0}^{t}\left\Vert \frac{\partial}{\partial x}I_{\epsilon,h(z)}(s,\,\cdot)\right\Vert _{L^{2}\left(\Omega^0;\,\tilde{L}_{0, w}^2\right)}^{2}ds \nonumber \\
& & \qquad\qquad +\xi^{2}\rho_2^2\int_{0}^{t}\left\Vert \frac{\partial}{\partial y}I_{\epsilon,1}(s,\,\cdot)\right\Vert _{L^{2}\left(\Omega^0;\,\tilde{L}_{0, w}^2\right)}^{2}ds \nonumber \\
& & \qquad\qquad + 2\xi\rho_3\rho_1
\rho_2\int_{0}^{t}\left\langle \frac{\partial}{\partial x}I_{\epsilon,h\left(z\right)}(s,\cdot),\,\frac{\partial}{\partial y}I_{\epsilon,1}(s,\cdot)\right\rangle _{L^2\left(\Omega^0; \, \tilde{L}_{0, w}^2\right)}ds
\end{eqnarray*}
which is the part of the identity of Lemma~\ref{mainid} that occurs from second order differentials when we apply Ito's formula, so $-N_5(t)$ is the counterpart of $\rho^2\int_0^t\int_0^{+\infty}\mathbb{E}[u_x^2(s, x)]dxds$ from \eqref{energy2}; Finally, 
\begin{eqnarray}\label{Ν6}
N_6(t) &=& -\int_{0}^{t}\int_{0}^{1}\int_{\mathbb{R}} \mathbb{E}\left[\left(h^2\left(Q^{-1}(v)\right)q\left(Q^{-1}(v)\right)u\left(s,\,x,\,Q^{-1}(v)\right)\right)_{x}\right. \nonumber \\
& & \qquad \qquad \qquad \qquad \quad \times \left.q\left(Q^{-1}(v)\right)u\left(s,\,x,\,Q^{-1}(v)\right)\right]dxdvds \nonumber \\
&=& -\frac{1}{2}\int_{0}^{t}\int_{\mathbb{R}} \mathbb{E}\left[(h^2\left(Q^{-1}(v)\right)q^2\left(Q^{-1}(v)\right)u^2\left(s,\,1,\,Q^{-1}(v)\right)\right]dvds \nonumber \\
&\leq& 0
\end{eqnarray}
is the $\epsilon \to 0$ limit of 
\begin{eqnarray*}
-\int_{0}^{t}\left\langle \frac{\partial}{\partial x}I_{\epsilon,h^{2}(z)}(s,\,\cdot),\,\mathbb{I}_{\left[0, \, 1\right] \times \mathbb{R}}(\cdot)I_{\epsilon,1}(s,\,\cdot)\right\rangle _{L^{2}\left(\Omega^0;\,\tilde{L}_{0}^2\right)}ds
\end{eqnarray*}
from the identity of Lemma~\ref{mainid}, which occurs due to the presence of the weight function $w^2(x)$ when we integrate by parts in the term that contains $u_{xx}(s,x,y)$, and $-N_6(t)$ is clearly the counterpart of $\int_0^t\mathbb{E}[u^2(s, 1)]ds$ from \eqref{energy2}.

Next we want control the quantities $N_i(t)$ for $i = 1,2,3,4,5,6$, and we start by using \eqref{corrcondition} and the standard inequality $2ab \leq a^2 + b^2$ to bound
\begin{eqnarray*}
& &2 \left(\xi\rho_{3}\rho_{1}\rho_{2} - \rho\right) \int_{0}^{t}\int_{\mathbb{R}^{+}}\int_{\mathbb{R}} w^2(x) \mathbb{E}\left[\left(h\left(Q^{-1}(v)\right)q\left(Q^{-1}(v)\right)u\left(s,\,x,\,Q^{-1}(v)\right)\right)_{x}\right. \\
& & \quad \qquad\qquad \qquad \quad\qquad\qquad \qquad \qquad\quad \times \left.\left(q\left(Q^{-1}(v)\right)u\left(s,\,x,\,Q^{-1}(v)\right)\right)_{v}\right]dxdvds \\
& & \qquad \leq \int_{0}^{t}\int_{\mathbb{R}^{+}}\int_{\mathbb{R}} w^2(x) \mathbb{E}\left[2\sqrt{1 - \rho^{2}_{1}}\left(h\left(Q^{-1}(v)\right)q\left(Q^{-1}(v)\right)u\left(s,\,x,\,Q^{-1}(v)\right)\right)_{x}\right. \\
& & \qquad \qquad\qquad \qquad \quad\qquad\qquad \times \left.\xi\sqrt{1 - \rho^{2}_{2}}\left(q\left(Q^{-1}(v)\right)u\left(s,\,x,\,Q^{-1}(v)\right)\right)_{v}\right]dxdvds \\
& & \qquad \leq \left(1 - \rho_{1}^{2}\right)\int_{0}^{t}\int_{\mathbb{R}^{+}}\int_{\mathbb{R}} w^2(x) \mathbb{E}\big[\big(h^2\left(Q^{-1}(v)\right) \nonumber \\
& & \qquad\qquad\qquad\qquad \qquad \qquad \qquad \qquad \times \big(q\left(Q^{-1}(v)\right)u\left(s,\,x,\,Q^{-1}(v)\right)\big)_{x}\big)^2\big]dxdvds \nonumber \\
& & \qquad \qquad + \xi^{2}\left(1 - \rho_{2}^{2}\right)\int_{0}^{t}\int_{\mathbb{R}^{+}}\int_{\mathbb{R}} w^2(x) \mathbb{E}\left[\left(\left(q\left(Q^{-1}(v)\right)u\left(s,\,x,\,Q^{-1}(v)\right)\right)_{v}\right)^2\right]dxdvds
\end{eqnarray*}
which gives
\begin{eqnarray}\label{N4N5}
& &N_4(t) + N_5(t) \leq -\xi^{2}\int_{0}^{t}\int_{\mathbb{R}^{+}}\int_{\mathbb{R}} w^2(x)\mathbb{E}\left[q'\left(Q^{-1}(v)\right)q\left(Q^{-1}(v)\right)u\left(s,\,x,\,Q^{-1}(v)\right)\right. \nonumber \\
& & \qquad \qquad \qquad \qquad \qquad\qquad\qquad\qquad \qquad \times \left.\left(q\left(Q^{-1}(v)\right)u\left(s,\,x,\,Q^{-1}(v)\right)\right)_{v}\right]dxdvds \nonumber \\
\end{eqnarray}
Then, to deal with $N_1(t)$, we treat it as its counterpart in \eqref{energy2}, so we fix a $T > 0$ and work for $t \in \left[0, \, T\right]$ and we use Girsanov's theorem to add a drift $\tilde{r}$ to the Brownian motion $W^{0}$. This adds the term $-\int_0^t\tilde{r}\rho_1h(y)(u(s,x,y))_xds$ to the SPDE~\eqref{eq:5.1}, which can be absorbed by the term $\frac{1}{2}\int_{0}^{t}h^{2}(y)\left(u(s,x,y)\right)_{x}ds$, affecting only $N_2(t)$ which is replaced with
\begin{eqnarray}
& & \int_{0}^{t} \int_{\mathbb{R}^{+}}\int_{\mathbb{R}} w^2(x)\mathbb{E}\left[\left(\left(h^2\left(Q^{-1}(v)\right) - 2\tilde{r}\rho_1h\left(Q^{-1}(v)\right)\right)q\left(Q^{-1}(v)\right)u\left(s,\,x,\,Q^{-1}(v)\right)\right)_{x}\right. \nonumber \\
& & \qquad \qquad  \qquad  \qquad  \qquad \qquad\qquad\qquad\qquad\quad \times \left.q\left(Q^{-1}(v)\right)u\left(s,\,x,\,Q^{-1}(v)\right)\right]dxdvds \nonumber \\
& & \quad = \mathbb{E}\left[\int_{0}^{t} \int_{\mathbb{R}^{+}}\int_{\mathbb{R}} w^2(x)\frac{h^2\left(Q^{-1}(v)\right)}{2}\left(q^2\left(Q^{-1}(v)\right)u^2\left(s,\,x,\,Q^{-1}(v)\right)\right)_{x}dxdvds\right] \nonumber \\
& & \quad \qquad - \mathbb{E}\left[\int_{0}^{t} \int_{\mathbb{R}^{+}}\int_{\mathbb{R}} w^2(x)\tilde{r}\rho_1h\left(Q^{-1}(v)\right)\left(q^2\left(Q^{-1}(v)\right)u^2\left(s,\,x,\,Q^{-1}(v)\right)\right)_{x}dxdvds\right] \nonumber \\
& & \quad = -\mathbb{E}\left[\int_{0}^{t} \int_{0}^{1}\int_{\mathbb{R}} \frac{h^2\left(Q^{-1}(v)\right)}{2}q^2\left(Q^{-1}(v)\right)u^2\left(s,\,x,\,Q^{-1}(v)\right)dxdvds\right] \nonumber \\
& & \quad \qquad +\mathbb{E}\left[\int_{0}^{t} \int_{0}^{1}\int_{\mathbb{R}} \tilde{r}\rho_1h\left(Q^{-1}(v)\right)q^2\left(Q^{-1}(v)\right)u^2\left(s,\,x,\,Q^{-1}(v)\right)dxdvds\right] \nonumber \\
& & \quad \leq -r\mathbb{E}\left[\int_{0}^{t} \int_{0}^{1}\int_{\mathbb{R}}q^2\left(Q^{-1}(v)\right)u^2\left(s,\,x,\,Q^{-1}(v)\right)dxdvds\right] \nonumber \\
& & \quad = -N_1(t)
\end{eqnarray}
where the last inequality follows by picking $\tilde{r}$ such that $\tilde{r}\rho_1 \leq -\frac{r}{m}$ for a lower bound $m$ of $h$. Thus we obtain
\begin{eqnarray}\label{N1N2}
N_1(t) + N_2(t) \leq 0.
\end{eqnarray}
Finally, writing 
\begin{eqnarray*}
&&q\left(Q^{-1}(v)\right)u\left(s,\,x,\,Q^{-1}(v)\right)\left(q\left(Q^{-1}(v)\right)u\left(s,\,x,\,Q^{-1}(v)\right)\right)_v \\
&& \qquad \qquad \qquad \qquad \qquad \qquad \qquad \qquad \qquad = \frac{1}{2}\left(q^2\left(Q^{-1}(v)\right)u^2\left(s,\,x,\,Q^{-1}(v)\right)\right)_v
\end{eqnarray*}
and using integration by parts in both \eqref{N3} and the right hand side of \eqref{N4N5}, we can bound
\begin{eqnarray}\label{N3N4N5}
& & N_3(t) + N_4(t) + N_5(t) \nonumber \\
& & \qquad \qquad \qquad \leq K\int_{0}^{t}\int_{\mathbb{R}^{+}}\int_{\mathbb{R}} w^2(x)\mathbb{E}\left[q^2\left(Q^{-1}(v)\right)u^2\left(s,\,x,\,Q^{-1}(v)\right)\right]dxdvds \nonumber \\
\end{eqnarray}
for some $K$ that depends only on the upper bounds of the derivatives of $Q'\left(Q^{-1}(v)\right)$, $q'\left(Q^{-1}(v)\right)$ and $vQ'\left(Q^{-1}(v)\right)$ (which are all bounded when $q$ satisfies the conditions of Assumption~\ref{ass1}), so plugging \eqref{N1N2}, \eqref{N3N4N5} and \eqref{Ν6} in \eqref{limid} we obtain:
\begin{eqnarray}
& & \int_{\mathbb{R}^{+}}\int_{\mathbb{R}} w^2(x)\mathbb{E}\left[q^2\left(Q^{-1}(v)\right)u^2\left(t,\,x,\,Q^{-1}(v)\right)\right]dxdv \nonumber \\
& & \qquad \qquad \leq K\int_{0}^{t}\int_{\mathbb{R}^{+}}\int_{\mathbb{R}} w^2(x)\mathbb{E}\left[q^2\left(Q^{-1}(v)\right)u^2\left(s,\,x,\,Q^{-1}(v)\right)\right]dxdvds, \nonumber 
\end{eqnarray}
for all $t \in [0, T]$. By Gronwall's inequality, the above can only hold when the always nonnegative quantity $\int_{\mathbb{R}^{+}}\int_{\mathbb{R}} w^2(x)\mathbb{E}\left[q^2\left(Q^{-1}(v)\right)u^2\left(t,\,x,\,Q^{-1}(v)\right)\right]dxdv$ equals $0$ for all $t \leq T$. However, we have used Girsanov's theorem to switch to a probability measure $\mathbb{Q}$ that gives a convenient drift to the Brownian motion $W^0$. Therefore, if we denote by $\mathbb{E}^{\mathbb{Q}}$ the expectation under $\mathbb{Q}$, our last conclusion gives: 
\begin{eqnarray}
&&\mathbb{E}^{\mathbb{Q}}\left[\int_0^T\int_{\mathbb{R}^{+}}\int_{\mathbb{R}} w^2(x)q^2\left(Q^{-1}(v)\right)u^2\left(t,\,x,\,Q^{-1}(v)\right)dxdvdt\right] = 0 \nonumber \\
&& \qquad\qquad \Rightarrow \mathbb{Q}\left(\int_0^T\int_{\mathbb{R}^{+}}\int_{\mathbb{R}} w^2(x)q^2\left(Q^{-1}(v)\right)u^2\left(t,\,x,\,Q^{-1}(v)\right)dxdvdt = 0\right) = 1 \nonumber \\
&& \qquad\qquad \Rightarrow \mathbb{P}\left(\int_0^T\int_{\mathbb{R}^{+}}\int_{\mathbb{R}} w^2(x)q^2\left(Q^{-1}(v)\right)u^2\left(t,\,x,\,Q^{-1}(v)\right)dxdvdt = 0\right) = 1 \nonumber \\ \label{finaluniqueness}
\end{eqnarray}
In the above display, the second equality follows from the first because $$\int_0^T\int_{\mathbb{R}^{+}}\int_{\mathbb{R}} w^2(x)q^2\left(Q^{-1}(v)\right)u^2\left(t,\,x,\,Q^{-1}(v)\right)dxdv$$ is always non-negative, while the third equality follows from the second because the probability measure $\mathbb{Q}$ given by Girsanov's theorem is equivalent to the original probability measure $\mathbb{P}$. Since $u$ is the difference of two arbitrary $\alpha$-solutions to our problem with the same parameters and the same initial data, \eqref{finaluniqueness} yields pathwise uniqueness of solutions in the space $L^{2}\left(\left[0,\,T\right];\,\tilde{L}^2_{0,w}\right)$. We have thus proven the following result:
\begin{thm}
Fix the function $h$, the real number $\rho$ and the initial data function $U_{0}$, and suppose that $q$ satisfies the conditions of Assumption~\ref{ass1}, $h$ is as in Theorem~\ref{thm:4.2}, and that we also have $\left|\rho - \xi\rho_{3}\rho_{1}\rho_{2}\right| \leq \xi\sqrt{1 - \rho^{2}_{1}}\sqrt{1 - \rho^{2}_{2}}$. Then, there exists a unique $u$ which is an $\alpha$-solution to our problem for all $\alpha \geq 0$, where we have pathwise uniqueness in the space $L^{2}\left(\left[0,\,T\right];\,\tilde{L}^2_{0,w}\right)$. For this $u$, the weak derivative $u_{y}$ exists and we have
\[ 
u_{y}\in L^{2}\left(\left[0,\,T\right]\times\Omega^0;\,\tilde{L}^2_{0,w}\right)
\]
\end{thm}

\begin{rem}
The best possible result would be to have our solution in the even smaller function space $L^2\left(\left[0, \, T\right]\times \Omega^0 ; H_{0}^1\left(\mathbb{R}^{+}\times\mathbb{R}\right)\right)$ (without the weight function $w^2(\cdot)$), which would be the case if we had an existence result in that space. However, standard Theorems that could give such an existence result are not applicable to our SPDE due to the unboundedness of the term $k(\theta - y)$ (see \cite{KR81}). This is also the reason we cannot deduce higher regularity immediately, as we have mentioned in Remark~\ref{2ndorderrem}. A possible solution to these problems could be to introduce mean-reverting volatility processes $\sigma_t^{i}$ having a bounded drift, i.e $q_2(\theta - \sigma_t^{i})$ for all $i \in \mathbb{N}$, for some increasing and bounded function $q_2$ vanishing at zero.
\end{rem}

\vspace*{.05in}

{\flushleft{\textbf{Acknowledgement}}\\[.1in]
The second author's work was supported financially by the United Kingdom Engineering and Physical Sciences Research Council {[}EP/L015811/1{]}}, and by the Foundation for Education and European Culture (founded by Nicos \& Lydia Tricha).

{\flushleft{\textbf{Data availability statement}}\\[.1in]
There is no data associated with this article}

\appendix

\section{Proofs of certain important results}

\subsection{Proof of Theorem~\ref{thm:2.6}}\label{pf2.6}
Since $f\left(0,\,y\right)=0$ for all $y$, Ito's formula for the stopped two-dimensional stochastic process $\left\{ \left(X_{t}^{1},\,\sigma_{t}^{1}\right):\,t\geq0\right\}$ implies that

\begin{eqnarray*}
& & f\left(X_{t\wedge T_{1}}^{1},\,\sigma_{t}^{1}\right) \\
& & \qquad = f\left(x^{1},\,\sigma^{1}\right) \\
& & \qquad\qquad +\int_{0}^{t}\left[f_{x}\left(X_{s}^{1},\,\sigma_{s}^{1}\right)\left(r-\frac{h^{2}\left(\sigma_{s}^{1}\right)}{2}\right)+kf_{y}\left(X_{s}^{1},\,\sigma_{s}^{1}\right)\left(\theta-\sigma_{s}^{1}\right)\right]\mathbb{I}_{\{T_{1}>s\}}ds \\
& &  \qquad\qquad+\frac{1}{2}\int_{0}^{t}\left[f_{xx}\left(X_{s}^{1},\,\sigma_{s}^{1}\right)h^{2}\left(\sigma_{s}^{1}\right)+\xi^{2}f_{yy}\left(X_{s}^{1},\,\sigma_{s}^{1}\right)q^2\left(\sigma_{s}^{1}\right)\right]\mathbb{I}_{\{T_{1}>s\}}ds \\
& &  \qquad\qquad+\xi\rho_{3}\rho_{1}\rho_{2}\int_{0}^{t}f_{xy}\left(X_{s}^{1},\,\sigma_{s}^{1}\right)h\left(\sigma_{s}^{1}\right)q\left(\sigma_{s}^{1}\right)\mathbb{I}_{\{T_{1}>s\}}ds \\
& &  \qquad\qquad+\int_{0}^{t}f_{x}\left(X_{s}^{1},\,\sigma_{s}^{1}\right)\mathbb{I}_{\{T_{1}>s\}}h\left(\sigma_{s}^{1}\right)\rho_{1}dW_{s}^{0} \\
& & \qquad\qquad+\xi\int_{0}^{t}f_{y}\left(X_{s}^{1},\,\sigma_{s}^{1}\right)\mathbb{I}_{\{T_{1}>s\}}q\left(\sigma_{s}^{1}\right)\rho_{2}dB_{s}^{0}
\\
& &  \qquad\qquad+\int_{0}^{t}f_{x}\left(X_{s}^{1},\,\sigma_{s}^{1}\right)\mathbb{I}_{\{T_{1}>s\}}h(\sigma_{s}^{1})\sqrt{1-\rho_{1}^{2}}dW_{s}^{1}
\\
& &  \qquad\qquad+\xi\int_{0}^{t}f_{y}\left(X_{s}^{1},\,\sigma_{s}^{1}\right)\mathbb{I}_{\{T_{1}>s\}}q\left(\sigma_{s}^{1}\right)\sqrt{1-\rho_{2}^{2}}dB_{s}^{1}.
\end{eqnarray*}
Due to the independence between the market and the idiosyncratic Brownian motions, taking conditional expectations given $\left(W^{0},\,B^{0}\right)$ and $\mathcal{G}^0$ eliminates the Ito integrals with respect to $B^1$ and $W^1$, while the Ito integrals with respect to $B^{0}$ and $W^{0}$ can be interchanged with the the conditional expectations. Interchanging then the Riemannian integrals with the conditional expectations and recalling the linearity of the integrals along with the definitions of $A$ and the integral with respect to $v_{t}$, we can easily obtain the desired, completing the proof of Theorem~\ref{thm:2.6}.

\subsection{Proof of Theorem~\ref{thm:3.1}}\label{pf3.3}
First, we will need sufficient integrability of our volatility processes. For this reason, we will first prove the following lemma

\begin{lem}\label{lem:3.2}
Fix a $T > 0$. Then, under the assumptions of Theorem~\ref{thm:3.1}, for any $p\geq0$ we have
\[
\mathbb{E}\left[\sup_{0\leq t\leq T}\sigma_{t}^{p}\right]<\infty.
\]
\end{lem}

\begin{proof}[\textbf{Proof of Lemma A.1}]
Observe that in all cases, the function $q$ satisfies the growth condition $q^2(z) \leq C_q^2|z|^{r}$ for some $r \in \{0, \, 1\}$ and some $C_q > 0$. It is enough to consider the case $p=2n$ with $n\in\mathbb{N}$. By Ito's formula we have
\begin{eqnarray}\label{itomonomial}
\sigma_{t}^{2n}&=& \sigma_{0}^{2n}+\int_{0}^{t}n\left(\sigma_{s}^{2n-1}2k\left(\theta-\sigma_{s}\right)+\xi^{2}\left(2n-1\right)\sigma_{s}^{2n-2}q^{2}\left(\sigma_{s}\right)\right)ds
\nonumber \\
&&\qquad +2n\xi\int_{0}^{t}\sigma_{s}^{2n-1}q\left(\sigma_{s}\right)d\left(\sqrt{1-\rho_{2}^{2}}B_{s}^{1}+\rho_{2}B_{s}^{0}\right) \nonumber
\\
&\leq& \sigma_{0}^{2n}+C_nT+2n\xi\int_{0}^{t}\sigma_{s}^{2n-1}q\left(\sigma_{s}\right)d\left(\sqrt{1-\rho_{2}^{2}}B_{s}^{1}+\rho_{2}B_{s}^{0}\right)
\end{eqnarray}
for some $C_n > 0$ and any $t\leq T$, since we can use the at most linear growth of $q^2$ to show that the quantity within the
Riemannian integral can be bounded from above by $P\left(|\sigma_s|\right)$, where $P$ is a polynomial of an even degree and a negative leading coefficient. Taking expectations on \eqref{itomonomial}, we find that for any $t \leq T$, the $2n$-th moment of $\sigma_{t}$ is bounded by $C_nT + \mathbb{E}\left[\sigma_{0}^{2n}\right]$. Moreover, by using this result and the inequality $2a \leq a^2 + 1$ for $a \geq 0$, we can prove that there is a uniform bound in $t \leq T$ for odd moments of $\sigma_{t}$ as well. Thus, if we set $M_{T}={\displaystyle \sup_{t\leq T}\sigma_{t}}$, taking the supremum over $t\leq T$ in \eqref{itomonomial}, then expectations and finally
using Cauchy-Schwartz and Doob's inequalities, we find that:
\begin{eqnarray}
\mathbb{E}\left[M_{T}^{2n}\right] &\leq & \mathbb{E}\left[\sigma_{0}^{2n}\right]+C_{n}T \nonumber \\
&& \qquad +2n\xi\mathbb{E}^{\frac{1}{2}}\left[\sup_{0\leq t\leq T}\left(\int_{0}^{t}\sigma_{s}^{2n-1}q\left(\sigma_{s}\right)d\left(\sqrt{1-\rho_{2}^{2}}B_{s}^{1}+\rho_{2}B_{s}^{0}\right)\right)^{2}\right] \nonumber \\
&\leq &\mathbb{E}\left[\sigma_{0}^{2n}\right]+C_{n}T+4n\xi\mathbb{E}^{\frac{1}{2}}\left[\int_{0}^{T}\sigma_{s}^{4n-2}q^2\left(\sigma_{s}\right) ds\right] \nonumber \\
&\leq &\mathbb{E}\left[\sigma_{0}^{2n}\right]+C_{n}T+4n\xi C_q\mathbb{E}^{\frac{1}{2}}\left[\int_{0}^{T}\sigma_{s}^{4n-2}\left|\sigma_{s}\right|^r ds\right] \nonumber \\
&=& \mathbb{E}\left[\sigma_{0}^{2n}\right]+C_{n}T+4n\xi C_q\sqrt{\int_{0}^{T}\mathbb{E}\left[\left|\sigma_{s}\right|^{4n-2 + r}\right]ds} \nonumber \\
&<& \infty \nonumber
\end{eqnarray}
and this completes the proof of the Lemma.
\end{proof}

Returning to the proof of our theorem, we need a few results from Malliavin Calculus. As in \cite{Nualart}, we denote by $\mathbb{D}^{n,p}(V)$ the space of random variables which are $n$ times Malliavin differentiable with respect to a Brownian motion defined in some interval $\left[0, \, T\right]$, which take values in the Banach space $V$, and whose $k$-th Malliavin derivative has an $L^2\left(\left[0, \, T\right]^k; V\right)$ norm which is a random variable in $L^p$ for all $0 \leq k \leq n$.

We will show that $\sigma_t$ possesses regular first and second order Malliavin derivatives with respect to the Brownian motion $B^1$, under the conditional probability measure $\mathbb{P}(\cdot\,|\,B^{0},\,\mathcal{G}^0)$, $\mathbb{P}^0$ - almost surely, which will allow us to obtain the desired result. 

We fix a $T > 0$ and we consider the strictly increasing function $Q(y) = \int_{0}^{y}\frac{1}{q(z)}dz$. Then, by using Ito's formula on \eqref{eq:3.1}, we find that $v_t = Q(\sigma_t)$ satisfies the SDE
\begin{eqnarray}\label{eq:3.1.2}
dv_{t} = V\left(v_t\right)dt + \xi\sqrt{1-\rho_{2}^{2}}dB_{t}^{1}+\xi \rho_{2}dB_{t}^{0}
\end{eqnarray}
for all $t \in \left[0, \, T\right]$, where
\begin{eqnarray}
V(x) = k\frac{(\theta-Q^{-1}\left(x\right))}{q\left(Q^{-1}\left(x\right)\right)} - \frac{\xi^2}{2}q'\left(Q^{-1}\left(x\right)\right)
\end{eqnarray} 
for any $x \in \mathbb{R}$. We will compute first the Malliavin derivative of $v_{t}$ with respect to the Brownian motion $B^1$, under the probability measure $\mathbb{P}(\cdot\,|\,B^{0},\,\mathcal{G}^0)$. Since we have $\left(Q^{-1}(x)\right)' = q\left(Q^{-1}(x)\right)$, we can easily compute
\begin{eqnarray}
V'(x) = q\left(Q^{-1}(x)\right) \left(k\frac{-q\left(Q^{-1}(x)\right) - \left(\theta - Q^{-1}(x)\right)q'\left(Q^{-1}(x)\right)}{q^2\left(Q^{-1}(x)\right)} - \frac{\xi^2}{2}q''\left(Q^{-1}(x)\right)\right) \nonumber \\
\end{eqnarray}
and
\begin{eqnarray}
V''(x) &=& q\left(Q^{-1}(x)\right)q'\left(Q^{-1}(x)\right) \nonumber \\
&& \qquad \qquad \times \left[k\frac{-q\left(Q^{-1}(x)\right) - \left(\theta - Q^{-1}(x)\right)q'\left(Q^{-1}(x)\right)}{q^2\left(Q^{-1}(x)\right)} - \frac{\xi^2}{2}q''\left(Q^{-1}(x)\right)\right] \nonumber \\
&& \qquad + 2kq'\left(Q^{-1}(x)\right) - k(\theta - Q^{-1}(x))\left[q''\left(Q^{-1}(x)\right) - 2\frac{\left(q'\left(Q^{-1}(x)\right)\right)^2}{q\left(Q^{-1}(x)\right)}\right] \nonumber \\
&& \qquad - \frac{\xi^2}{2}q^{(3)}\left(Q^{-1}(x)\right)\left(q\left(Q^{-1}(x)\right)\right)^2
\end{eqnarray}
which are both bounded by Assumption~\ref{ass1}. This allows us to recall Theorem 2.2.1 from page 102 in \cite{Nualart} and the Remark after its proof, which imply that the desired Malliavin derivative exists under $\mathbb{P}$ and it satisfies
\begin{eqnarray}\label{malb}
D_{t'}v_{t} &=& \xi \sqrt{1 - \rho_2^2} - \int_{t'}^{t}V'(v_s)D_{t'}v_{s}ds
\end{eqnarray}
for $t \in \left[t', \, T\right]$. By looking at the proof of that Theorem, we see that the underlying probability measure does not play any role as long as we are differentiating with respect to the path of a Brownian motion, which means that here we can have the same result under the probability measure $\mathbb{P}(\cdot\,|\,B^{0},\,\mathcal{G}^0)$ as well. Then, \eqref{malb} is a linear ODE in $t \in \left[t', \, T\right]$, which can be solved to give
\begin{eqnarray}
D_{t'}v_{t} &=& \xi \sqrt{1 - \rho_2^2}e^{-\int_{t'}^{t}V'(v_s)ds}
\end{eqnarray}
for any $0 \leq t' \leq t \leq T$, while we have $D_{t'}v_{t} = 0$ for $0 \leq t < t' \leq T$. By the boundedness of $V'$, the above Malliavin derivative is positive, and for any $t \in \left[0, \, T\right]$, it belongs to $L^{\infty}\left(\Omega^0 \times \Omega^1 \times \left[0, \, T\right]\right)$. Therefore, we can use the standard Malliavin chain rule (Proposition 1.2.2 on page 29 in \cite{Nualart}) to obtain
\begin{eqnarray}\label{maldev1}
D_{t'}\sigma_t &=& D_{t'}Q^{-1}(v_t) \nonumber \\
&=& q\left(Q^{-1}(v_t)\right)D_{t'}v_{t} \nonumber \\
&=& \xi \sqrt{1 - \rho_2^2}q\left(Q^{-1}(v_t)\right)e^{-\int_{t'}^{t}V'(v_s)ds} \nonumber \\
&=& \xi \sqrt{1 - \rho_2^2}q(\sigma_t)e^{-\int_{t'}^{t}V'(Q(\sigma_s))ds},
\end{eqnarray}
which belongs to $L^{\infty}\left(\Omega^0 \times \Omega^1 \times \left[0, \, T\right]\right)$ as well.

To compute the second order derivative of $\sigma_t$ under the conditional probability measure $\mathbb{P}(\cdot\,|\,B^{0},\,\mathcal{G}^0)$, we use again the standard Malliavin chain rule to obtain 
\begin{eqnarray}
D_{t''}V'(v_s) &=& V''(v_s)D_{t''}v_s \nonumber \\
&=& \xi \sqrt{1 - \rho_2^2}V''(v_s)e^{-\int_{t''}^{s}V'(v_{s'})ds'},
\end{eqnarray}
for any $s \in \left[t', \, t\right]$ with $0 \leq t' \leq t \leq T$ and $t'' \in \left[0, \, T\right]$. It is not hard to see that the absolute value of this derivative is bounded by some $b_T > 0$ depending only on $T$. Thus, for any $h_{\cdot} \in L^2\left(\Omega^0 \times \Omega^1 \times \left[0, \, T\right]\right)$, if we denote by $\delta$ the Skorokhod integral with respect to $B^1$, by Fubini's Theorem and duality we have
\begin{eqnarray}
\mathbb{E}\left[\int_{0}^{t}h_{t''}\int_{t'}^{t}D_{t''}V'(v_s)dsdt''\,|\,B^{0},\,\mathcal{G}^0\right] &=& \int_{t'}^{t}\mathbb{E}\left[\int_{0}^{t}D_{t''}V'(v_s)h_{t''}dt''\,|\,B^{0},\,\mathcal{G}^0\right]ds \nonumber \\
&=& \int_{t'}^{t}\mathbb{E}\left[V'(v_s)\delta(h_{\cdot})\,|\,B^{0},\,\mathcal{G}^0\right]ds
\nonumber \\
&=& \mathbb{E}\left[\delta(h_{\cdot})\int_{t'}^{t}V'(v_s)ds\,|\,B^{0},\,\mathcal{G}^0\right]. \nonumber 
\end{eqnarray}
Therefore, if we approximate $\int_{t'}^{t}V'(v_s)ds$ by square integrable random variables which are Malliavin differentiable, we can show that the Malliavin derivatives converge in $L^2$ to the integral $\int_{t'}^{t}D_{t''}V'(v_s)ds$, and Lemma 1.2.3 on page 30 in \cite{Nualart} implies then that $\int_{t'}^{t}V'(v_s)ds$ has a Malliavin derivative which equals $\int_{t'}^{t}D_{t''}V'(v_s)ds$. Combining the last two results we find that
\begin{eqnarray}
D_{t''}\int_{t'}^{t}V'(v_s)ds = \xi \sqrt{1 - \rho_2^2}\int_{t'}^{t}V''(v_s)e^{-\int_{t''}^{s}V'(v_{s'})ds'}ds
\end{eqnarray}
which has an absolute value bounded by $Tb_T$. Moreover, since $\int_{t'}^{t}V'(v_s)ds$ is bounded for any $t$ and $t'$, by \eqref{maldev1} we have that
\begin{eqnarray}
D_{t'}\sigma_t = \xi \sqrt{1 - \rho_2^2}q\left(Q^{-1}(v_t)\right)e^{-F_{t,t'}\left(\int_{t'}^{t}V'(v_s)ds\right)} \nonumber 
\end{eqnarray}
for a smooth and compactly supported function $F_{t,t'}$ which is the identity on a compact set containing all the possible values of $\int_{t'}^{t}V'(v_s)ds$. Hence, applying the Malliavin chain rule once more, we obtain
\begin{eqnarray}
D_{t', t''}\sigma_{t} &=& \xi \sqrt{1 - \rho_2^2}q'\left(Q^{-1}(v_t)\right)q\left(Q^{-1}(v_t)\right)e^{-\int_{t'}^{t}V'(v_s)ds}D_{t''}v_t \nonumber \\
&& \qquad + \xi \sqrt{1 - \rho_2^2}q\left(Q^{-1}(v_t)\right)e^{-\int_{t'}^{t}V'(v_s)ds}D_{t''}\int_{t'}^{t}V'(v_s)ds \nonumber \\
&=& \xi^2 \left(1 - \rho_2^2\right)q'\left(Q^{-1}(v_t)\right)q\left(Q^{-1}(v_t)\right)e^{-\int_{t'}^{t}V'(v_s)ds}e^{-\int_{t''}^{t}V'(v_s)ds} \nonumber \\
&& \qquad + \xi^2 \left(1 - \rho_2^2\right)q\left(Q^{-1}(v_t)\right)e^{-\int_{t'}^{t}V'(v_s)ds}\int_{t'}^{t}V''(v_s)e^{-\int_{t''}^{s}V'(v_{s'})ds'}ds \nonumber \\
&=& \xi^2 \left(1 - \rho_2^2\right)q'\left(\sigma_t\right)q\left(\sigma_t\right)e^{-\int_{t'}^{t}V'(Q(\sigma_s))ds}e^{-\int_{t''}^{t}V'(Q(\sigma_s))ds} \nonumber \\
&& \qquad + \xi^2 \left(1 - \rho_2^2\right)q\left(\sigma_t\right)e^{-\int_{t'}^{t}V'(Q(\sigma_s))ds}\int_{t'}^{t}V''(Q(\sigma_s))e^{-\int_{t''}^{s}V'(Q(\sigma_{s'}))ds'}ds. \nonumber \\
\end{eqnarray}
By our previous boundedness results, the last line also has an absolute value which is bounded by some $b_T' > 0$ depending only on $T$. 

By Lemma~\ref{lem:3.2} now, we have that $\sigma_{t}\in L^{p}\left(\Omega^0 \times \Omega^1\right)$ for any $p > 1$ and $t \leq T$, under the probability measure $\mathbb{P}$, which implies that $\sigma_{t}\in L^{p}\left(\Omega^1\right)$ holds under the conditional probability measure $\mathbb{P}\left(\cdot\,|\,B^{0},\,\mathcal{G}^0\right)$, $\mathbb{P}^0$- almost surely. Combining this with the above results, we deduce that $\sigma_{t}\in L^{p}\left(\Omega^1\right) \cap \mathbb{D}^{\infty,2}\left(\Omega^1\right)$ with respect to $B^1$ for all $t \leq T$, under the probability measure $\mathbb{P}\left(\cdot\,|\,B^{0},\,\mathcal{G}^0\right)$ ($\mathbb{P}^0$- almost surely). Moreover, it is not hard to check that for all $t' < t$, $D_{t'}\sigma_t$ is bounded between the positive quantities $b_T'' := \xi \sqrt{1 - \rho_2^2}m_qe^{-TM_{|V'|}}$ and $\tilde{b}_T := \xi \sqrt{1 - \rho_2^2}M_qe^{TM_{|V'|}}$, where we denote by $m_f$ the minimum of a function $f$ and by $M_f$ the maximum of that function.
This allows us to recall Lemma~E2.1 from \cite{ERR} for $\alpha = 0$ and deduce that the conditional probability measure $\mathbb{P}(\sigma_{t}\in \cdot\,|\,B^{0},\,\mathcal{G}^0)$ has a continuous density $p_{t}(\cdot\,|\,B^{0},\,\mathcal{G}^0)$ which satisfies
\begin{eqnarray}\label{3.1est}
&& \sup_{y \in \mathbb{R}}p_{t}(y\,|\,B^{0},\,\mathcal{G}^0) \nonumber \\
&& \qquad \leq \left(C+2\right)\mathbb{E}^{\frac{1}{pr}}\left[\left\Vert D_{\cdot, \cdot}^2\sigma_t \right\Vert_{L^2\left(\left[0, T\right]^2\right)}^{pr}\,|\,B^{0},\,\mathcal{G}^0\right] \nonumber \\
&& \qquad \qquad \qquad \times\mathbb{E}^{\frac{1}{pr'}}\left[\left|\frac{1}{\left\Vert D_{\cdot}\sigma_t \right\Vert_{L^2\left(\left[0, T\right]\right)}^2}\right|^{pr'}\,|\,B^{0},\,\mathcal{G}^0\right] \nonumber \\
&& \qquad \qquad + C\mathbb{E}^{\frac{1}{pr}}\left[\left\Vert D_{\cdot}\sigma_t \right\Vert_{L^2\left(\left[0, T\right]\right)}^{pr}\,|\,B^{0},\,\mathcal{G}^0\right]\times\mathbb{E}^{\frac{1}{pr'}}\left[\left|\frac{1}{\left\Vert D_{\cdot}\sigma_t \right\Vert_{L^2\left(\left[0, T\right]\right)}^2}\right|^{pr'} \,|\,B^{0},\,\mathcal{G}^0\right], \nonumber \\
&& \qquad = \left(C+2\right)\mathbb{E}^{\frac{1}{pr}}\left[\left\Vert D_{\cdot, \cdot}^2\sigma_t \right\Vert_{L^2\left(\left[0, t\right]^2\right)}^{pr}\,|\,B^{0},\,\mathcal{G}^0\right] \nonumber \\
&& \qquad \qquad \qquad \times\mathbb{E}^{\frac{1}{pr'}}\left[\left|\frac{1}{\left\Vert D_{\cdot}\sigma_t \right\Vert_{L^2\left(\left[0, t\right]\right)}^2}\right|^{pr'}\,|\,B^{0},\,\mathcal{G}^0\right] \nonumber \\
&& \qquad \qquad + C\mathbb{E}^{\frac{1}{pr}}\left[\left\Vert D_{\cdot}\sigma_t \right\Vert_{L^2\left(\left[0, t\right]\right)}^{pr}\,|\,B^{0},\,\mathcal{G}^0\right]\times\mathbb{E}^{\frac{1}{pr'}}\left[\left|\frac{1}{\left\Vert D_{\cdot}\sigma_t \right\Vert_{L^2\left(\left[0, t\right]\right)}^2}\right|^{pr'}\,|\,B^{0},\,\mathcal{G}^0\right], \nonumber \\
&& \qquad \leq \left(C+2\right)\mathbb{E}^{\frac{1}{pr}}\left[\left(\sqrt{t^2b_T'}\right)^{pr}\,|\,B^{0},\,\mathcal{G}^0\right]\times\mathbb{E}^{\frac{1}{pr'}}\left[\left(\frac{1}{tb_T''}\right)^{pr'}\,|\,B^{0},\,\mathcal{G}^0\right] \nonumber \\
&& \qquad \qquad + C\mathbb{E}^{\frac{1}{pr}}\left[\left(\sqrt{t\tilde{b}_T}\right)^{pr}\,|\,B^{0},\,\mathcal{G}^0\right]\times\mathbb{E}^{\frac{1}{pr'}}\left[\left(\frac{1}{tb_T''}\right)^{pr'}\,|\,B^{0},\,\mathcal{G}^0\right] \nonumber \\
&& \qquad = \left(C+2\right)\sqrt{t^2b_T'}\times\frac{1}{tb_T''} + C\sqrt{t\tilde{b}_T}\times\frac{1}{tb_T''}, \nonumber
\end{eqnarray}
where $C$ is a deterministic positive constant, $p, r$ can be anything bigger than $1$, and $\tilde{b}_T, b_T', b_T''$ are the deterministic bounds obtained earlier. The above now gives 
\begin{eqnarray*}
\esssup_{y \in \mathbb{R}, \, \omega^0 \in \Omega^0}p_{t}(y\,|\,B^{0},\,\mathcal{G}^0) \leq C' + C''\frac{1}{\sqrt{t}}
\end{eqnarray*}
for some deterministic constants $C'$ and $C''$, which can be integrated with respect to $t$ in $\left[0, \, T\right]$ to give
\begin{eqnarray*}
\int_{0}^{T}\esssup_{y \in \mathbb{R}, \, \omega^0 \in \Omega^0}p_{t}(y\,|\,B^{0},\,\mathcal{G}^0)dt \leq C'T + C''\sqrt{T}
\end{eqnarray*}
and this completes the proof of the first estimate. 

For the second estimate, we recall again Lemma~E2.1 from \cite{ERR} but for $a > 0$, and we use the same bounds as previously for the Malliavin derivatives to obtain:
\begin{eqnarray}
&& \sup_{y \in \mathbb{R}}|y|^{\alpha}p_{t}(y\,|\,B^{0},\,\mathcal{G}^0) \nonumber \\
&& \qquad \leq \left(C+2\right)\mathbb{E}^{\frac{1}{pr}}\left[\left(\sqrt{t^2b_T'}\right)^{pr}\,|\,B^{0},\,\mathcal{G}^0\right]\times\mathbb{E}^{\frac{1}{pr'}}\left[\left(\frac{1}{tb_T''}\right)^{pr'}|\sigma_t|^{\alpha}\,|\,B^{0},\,\mathcal{G}^0\right] \nonumber \\
&& \qquad \qquad + C\mathbb{E}^{\frac{1}{pr}}\left[\left(\sqrt{t\tilde{b}_T}\right)^{pr}\,|\,B^{0},\,\mathcal{G}^0\right]\times\mathbb{E}^{\frac{1}{pr'}}\left[\left(\frac{1}{tb_T''}\right)^{pr'}|\sigma_t|^{\alpha}\,|\,B^{0},\,\mathcal{G}^0\right]. \nonumber
\end{eqnarray}
Therefore, by Holder's inequality we have
\begin{eqnarray}
&& \mathbb{E}\left[\left(\sup_{y \in \mathbb{R}}|y|^{\alpha}p_{t}(y\,|\,B^{0},\,\mathcal{G}^0)\right)^p\right] \nonumber \\
&& \qquad \leq \left(C+2\right)\mathbb{E}^{\frac{1}{r}}\left[\left(\sqrt{t^2b_T'}\right)^{pr}\right]\times\mathbb{E}^{\frac{1}{r'}}\left[\left(\frac{1}{tb_T''}\right)^{pr'}|\sigma_t|^{\alpha}\right] \nonumber \\
&& \qquad \qquad + C\mathbb{E}^{\frac{1}{r}}\left[\left(\sqrt{t\tilde{b}_T}\right)^{pr}\right]\times\mathbb{E}^{\frac{1}{r'}}\left[\left(\frac{1}{tb_T''}\right)^{pr'}|\sigma_t|^{\alpha}\right] \nonumber \\
&& \qquad \leq \left(\left(C+2\right)\left(\frac{\sqrt{b_T'}}{b_T''}\right)^{p} + Ct^{-\frac{p}{2}}\left(\frac{\sqrt{\tilde{b}_T}}{b_T''}\right)^{p}\right)\mathbb{E}^{\frac{1}{r'}}\left[\sup_{0 \leq s \leq T}|\sigma_s|^{\alpha}\right] \nonumber
\end{eqnarray}
whose integral in $t$ over $\left[0, \, T\right]$ is finite by Lemma~\ref{lem:3.2}, since $p < 2$. This completes the proof of Theorem~\ref{thm:3.1}.

\subsection{Proof of Lemma~\ref{lem:5.3}}\label{pf5.5}
We start by setting $z = Q^{-1}(v), v \in \mathbb{R}$, which allows us to write
\begin{eqnarray*}
J_{u,\epsilon}(\lambda,\,y)&=& \mathbb{\int_{\mathbb{R}}}u(\lambda,\,z)\frac{1}{\sqrt{2\pi\epsilon}}e^{-\frac{(Q(z)-y)^{2}}{2\epsilon}}dz
\\
&=& \int_{\mathbb{R}}q\left(Q^{-1}(v)\right)u\left(\lambda,\,Q^{-1}(v)\right)\frac{1}{\sqrt{2\pi\epsilon}}e^{-\frac{(v-y)^{2}}{2\epsilon}}dv \\
\\
&=& \int_{\mathbb{R}}J_{u}(\lambda,\,v)\frac{1}{\sqrt{2\pi\epsilon}}e^{-\frac{(v-y)^{2}}{2\epsilon}}dv
\end{eqnarray*}
We will show 2 first. Observe that by a standard property of the standard heat kernel, the desired convergence holds in $L^2\left(\mathbb{R}\right)$ for all $\lambda \in \Lambda$. Then, by another standard property of that kernel (which is obtained by a simple application of the Cauchy-Schwartz inequality) we have $\left\Vert J_{u,\epsilon}(\lambda,\,\cdot)\right\Vert_{L^2\left(\mathbb{R}\right)} \leq \left\Vert J_{u}(\lambda,\,\cdot)\right\Vert_{L^2\left(\mathbb{R}\right)}$, allowing us to bound $\left\Vert J_{u,\epsilon}(\lambda,\,\cdot) - J_{u}(\lambda,\,\cdot)\right\Vert_{L^2\left(\mathbb{R}\right)}$ by $2\left\Vert J_{u}(\lambda,\,\cdot)\right\Vert_{L^2\left(\mathbb{R}\right)}$ whichs belongs to the space $L^2\left(\Lambda\right)$. Thus, the dominated convergence theorem implies that the desired convergence holds also in $L^{2}\left(\left(\Lambda, \, \mu \right);\,L^{2}\left(\mathbb{R}
\right)\right)$. 

We proceed now to the proof of 1. By our regularity assumptions and the properties of the standard heat kernel we have that $J_{u,\epsilon}(\lambda,\,y)$ is smooth and it's $n$-th derivative in $y$ is equal to 
\begin{eqnarray*}
\int_{\mathbb{R}}q\left(Q^{-1}(v)\right)u\left(\lambda,\,Q^{-1}(v)\right)\frac{1}{\sqrt{2\pi\epsilon}}P(v - y)e^{-\frac{(v-y)^{2}}{2\epsilon}}dv,
\end{eqnarray*}
where $P$ is some polynomial of degree $n$. Therefore, it suffices to show that for any positive integer $n$, the quantity 
\begin{eqnarray*}
\int_{\Lambda}\int_{\mathbb{R}}\left(\int_{\mathbb{R}}q\left(Q^{-1}(v)\right)u\left(\lambda,\,Q^{-1}(v)\right)(v - y)^{n}\frac{e^{-\frac{(v-y)^{2}}{2\epsilon}}}{\sqrt{2\pi\epsilon}}dv\right)^{2}dyd\mu(\lambda)
\end{eqnarray*}
is finite. By the Cauchy-Schwartz inequality, the above is bounded by
\begin{eqnarray}
& & \int_{\Lambda}\int_{\mathbb{R}}\left(\int_{\mathbb{R}}q^2\left(Q^{-1}(v)\right)u^2\left(\lambda,\,Q^{-1}(v)\right)(v - y)^{2n}\frac{e^{-\frac{(v-y)^{2}}{2\epsilon}}}{\sqrt{2\pi\epsilon}}dv\right) \nonumber \\
& & \qquad \qquad \times\left(\int_{\mathbb{R}}\frac{e^{-\frac{(v-y)^{2}}{2\epsilon}}}{\sqrt{2\pi\epsilon}}dv\right)dyd\mu(\lambda) \nonumber \\
&&\qquad = \int_{\Lambda}\int_{\mathbb{R}}\int_{\mathbb{R}}q^2\left(Q^{-1}(v)\right)u^2\left(\lambda,\,Q^{-1}(v)\right)(y - v)^{2n}\frac{e^{-\frac{(v-y)^{2}}{2\epsilon}}}{\sqrt{2\pi\epsilon}}dvdyd\mu(\lambda) \nonumber \\
\end{eqnarray}
and thus, by Fubini's Theorem, it suffices to show that the integral
\begin{eqnarray*}
\int_{\Lambda}\int_{\mathbb{R}}q^2\left(Q^{-1}(v)\right)u^2\left(\lambda,\,Q^{-1}(v)\right)\left(\int_{\mathbb{R}}(y - v)^{2n}\frac{e^{-\frac{(v-y)^{2}}{2\epsilon}}}{\sqrt{2\pi\epsilon}}dy\right)dvd\mu(\lambda)
\end{eqnarray*}
is finite, which is obvious since $J_u\left(\lambda, \, v\right) = q\left(Q^{-1}(v)\right)u\left(\lambda,\,Q^{-1}(v)\right)$ is integrable and the inner integral converges. The proof of Lemma~\ref{lem:5.3} is now complete.

\subsection{Proof of Lemma~\ref{lem:5.4}}\label{pf5.6}
Almost identical to the proof of \cite[Lemma 5.5]{HK17}, with the only differences being that we take $\delta' = 0$, integrations over $\mathbb{R}^{+}$ now take place over $\mathbb{R}$, the square root function in the smoothing kernel is replaced by $Q$, and the function $J_u\left(\lambda, \, v\right)$ which appears when we use a transformation to convert the smoothing kernel to the standard heat kernel equals $q\left(Q^{-1}(v)\right)u\left(\lambda,\,Q^{-1}(v)\right)$ rather than $2vu\left(\lambda,\,v^2\right)$.

\subsection{Proof of Lemma~\ref{mainid}}\label{pf5.7}
First, the assumed weighted integrability of $u$ and $u_{x}$ and Lemma~\ref{lem:5.3} imply that all the terms in the identity we are proving are finite. Next, by the definition of $\phi_{\epsilon}$ we have
\begin{eqnarray}
\int_{\mathbb{R}}g(z)u(s,\,x,\,z)\frac{\partial}{\partial z}\phi_{\epsilon}(z,\,y)dz 
&=& -\int_{\mathbb{R}}Q'(z)g(z)u(s,\,x,\,z)\frac{\partial}{\partial y}
\phi_{\epsilon}(z,\,y)dz \nonumber \\
&=& -\frac{\partial}{\partial y} I_{\epsilon,g(z)Q'(z)}(s,\,x,\,y) \label{eq:5.8}
\end{eqnarray}
and
\begin{eqnarray}
& &  \int_{\mathbb{R}}g(z)u(s,\,x,\,z)\frac{\partial^{2}}{\partial z^{2}}\phi_{\epsilon}(z,\,y)dz \nonumber \\
& & \qquad = -\int_{\mathbb{R}}g(z)u(s,\,x,\,z)\frac{\partial}{\partial z}\left(\frac{\partial}{\partial y}\phi_{\epsilon}(z,\,y)Q'(z)\right)dz \nonumber \\
&& \qquad = -\int_{\mathbb{R}}g(z)u(s,\,x,\,z)\frac{\partial}{\partial y}\frac{\partial}{\partial z}\phi_{\epsilon}(z,\,y)Q'(z)dz \nonumber \\
&& \qquad \qquad-\int_{\mathbb{R}}g(z)u(s,\,x,\,z)\frac{\partial}{\partial y}\phi_{\epsilon}(z,\,y)Q''(z)dz \nonumber \\
& & \qquad = -\left(\int_{\mathbb{R}}g(z)Q'(z)u(s,\,x,\,z)\frac{\partial}{\partial z}\phi_{\epsilon}(z,\,y)
dz\right)_{y} \nonumber \\
&& \qquad \qquad - \int_{\mathbb{R}}g(z)Q''(z)u(s,\,x,\,z)\frac{\partial}{\partial y}\left(\phi_{\epsilon}(z,\,y)\right)dz \nonumber \\
& & \qquad = \left(I_{\epsilon,g(z)\left(Q'(z)\right)^2}(s,\,x,\,y)\right)_{yy} -\left(I_{\epsilon,g(z)Q''(z)}(s,\,x,\,y)\right)_{y},
\label{eq:5.9}
\end{eqnarray}
for any $\epsilon>0$ and function $g$. Therefore, testing
\eqref{eq:5.1} against $\phi_{\epsilon}$, substituting from \eqref{eq:5.8} and \eqref{eq:5.9}, and taking $x$-derivatives outside the integrals in $z$ we find that
\begin{eqnarray}
I_{\epsilon,1}(t,\,x,\,y)&=& \int_{\mathbb{R}}U_{0}(x,\,z)\phi_{\epsilon}(z,\,y)dz-r\int_{0}^{t}\frac{\partial}{\partial x}I_{\epsilon,1}(s,\,x,\,y)ds \nonumber \\
& & \quad +\frac{1}{2}\int_{0}^{t}\frac{\partial}{\partial x}I_{\epsilon,h^{2}(z)}(s,\,x,\,y)ds -k\theta\int_{0}^{t}\frac{\partial}{\partial y}I_{\epsilon,Q'(z)}(s,\,x,\,y)ds \nonumber \\
& & \quad +k\int_{0}^{t}\frac{\partial}{\partial y}I_{\epsilon,zQ'(z)}(s,\,x,\,y)ds +\frac{1}{2}\int_{0}^{t}\frac{\partial^{2}}{\partial x^{2}}I_{\epsilon,h^{2}(z)}(s,\,x,\,y)ds \nonumber \\
& & \quad +\frac{\xi^{2}}{2}\int_{0}^{t}\frac{\partial^{2}}{\partial y^{2}}I_{\epsilon,1}(s,\,x,\,y)ds +\frac{\xi^{2}}{2}\int_{0}^{t}\frac{\partial}{\partial y}I_{\epsilon,q'(z)}(s,\,x,\,y)ds \nonumber \\
& & \quad +\rho\int_{0}^{t}\frac{\partial^{2}}{\partial x \partial y}I_{\epsilon,h\left(z\right)}(s,\,x,\,y)ds -\rho\int_{0}^{t}\frac{\partial}{\partial x}I_{\epsilon,h(z)}(s,\,x,\,y)dW_{s}^{0} \nonumber \\
& & \quad -\xi\rho_{2}\int_{0}^{t}\frac{\partial}{\partial y}I_{\epsilon,1}(s,\,x,\,y)dB_{s}^{0}. \label{eq:5.10}
\end{eqnarray}
If we multiply the above by $w^2(x)$, apply Ito's formula for the $L^{2}(\mathbb{R}^{+})$ norm (Theorem~3.1 from \cite{KR81} for the triple $H_{0}^{1}\subset L^{2}\subset H^{-1}$, with $\Lambda(u) = w(\cdot)u$), and then integrate in
$y$ over $\mathbb{R}$, we obtain
\begin{eqnarray}
& & \left\Vert I_{\epsilon,1}(t,\,\cdot)\right\Vert _{\tilde{L}_{w}^2}^{2}=\left\Vert \int_{\mathbb{R}}U_{0}(\cdot,z)
\phi_{\epsilon}(z,\cdot)dz\right\Vert _{\tilde{L}_{0,w}^2}^{2} \nonumber \\
& & \qquad -2r\int_{0}^{t}\left\langle \frac{\partial}{\partial x}I_{\epsilon,1}(s,\cdot),I_{\epsilon,1}(s,\cdot)\right\rangle _{\tilde{L}_{0,w}^2}ds+\int_{0}^{t}\left\langle \frac{\partial}{\partial x}I_{\epsilon,h^{2}(z)}(s,\cdot),I_{\epsilon,1}(s,\cdot)\right\rangle _{\tilde{L}_{0,w}^2}ds \nonumber \\
& & \qquad -2k\theta\int_{0}^{t}\left\langle \frac{\partial}{\partial y}I_{\epsilon, Q'(z)}(s,\cdot),I_{\epsilon,1}(s,\cdot)\right\rangle _{\tilde{L}_{0,w}^2}ds \nonumber \\
& & \qquad +2k\int_{0}^{t}\left\langle \frac{\partial}{\partial y}I_{\epsilon,zQ'(z)}(s,\cdot),I_{\epsilon,1}(s,\cdot)\right\rangle _{\tilde{L}_{0,w}^2}ds \nonumber \\
& & \qquad +\int_{0}^{t}\left\langle \frac{\partial^{2}}{\partial x^{2}}I_{\epsilon,h^{2}(z)}(s,\cdot),\,I_{\epsilon,1}(s,\cdot)\right\rangle _{\tilde{L}_{0,w}^2}ds \nonumber \\
& & \qquad +\xi^{2}\int_{0}^{t}\left\langle \frac{\partial^{2}}{\partial y^{2}}I_{\epsilon,1}(s,\cdot),
I_{\epsilon,1}(s,\cdot)\right\rangle _{\tilde{L}_{0,w}^2}ds \nonumber \\
& & \qquad +2\rho\int_{0}^{t}\left\langle \frac{\partial^{2}}{\partial x \partial y}I_{\epsilon,h\left(z\right)}(s,\cdot),\,I_{\epsilon,1}(s,\cdot)\right\rangle _{\tilde{L}_{0,w}^2}ds \nonumber \\
& & \qquad +2\xi\rho_{3}\rho_{1}\rho_{2}\int_{0}^{t}\left\langle \frac{\partial}{\partial x}I_{\epsilon,h\left(z\right)}(s,\cdot),\,\frac{\partial}{\partial y}I_{\epsilon,1}(s,\cdot)\right\rangle _{\tilde{L}_{0,w}^2}ds \nonumber \\
& & \qquad +\xi^{2}\int_{0}^{t}\left\langle \frac{\partial}{\partial y}I_{\epsilon, q'(z)}(s,\cdot),I_{\epsilon,1}(s,\,\cdot)\right\rangle _{\tilde{L}_{0,w}^2}ds+\rho_{1}^{2}\int_{0}^{t}\left\Vert \frac{\partial}{\partial x}
I_{\epsilon,h(z)}(s,\cdot)\right\Vert _{\tilde{L}_{0,w}^2}^{2}ds \nonumber \\
& & \qquad +\xi^{2}\rho_{2}^{2}\int_{0}^{t}\left\Vert \frac{\partial}{\partial y}
I_{\epsilon,1}(s,\cdot)\right\Vert _{\tilde{L}_{0,w}^2}^{2}ds
\nonumber \\
& & \qquad -2\rho_{1}\int_{0}^{t}\left\langle \frac{\partial}{\partial x}I_{\epsilon,h(z)}(s,\cdot),I_{\epsilon,1}(s,\cdot)\right\rangle _{\tilde{L}_{0,w}^2}dW_{s}^{0} \nonumber \\
& &\qquad -2\xi\rho_{2}\int_{0}^{t}\left\langle \frac{\partial}{\partial y}I_{\epsilon,1}(s,\cdot),I_{\epsilon,1}(s,\cdot)\right\rangle _{\tilde{L}_{0,w}^2}dB_{s}^{0}, \label{eq:5.11}
\end{eqnarray}
where an integration by parts yields:
\begin{eqnarray}\label{simpleIBP}
-2r\left\langle \frac{\partial}{\partial x}I_{\epsilon,1}(s,\cdot),I_{\epsilon,1}(s,\cdot)\right\rangle _{\tilde{L}_{0,w}^2} &=& -2r\int_{\mathbb{R}}\int_0^{+\infty}w^2(x)\left(I_{\epsilon, 1}(s, x, y)\right)_xI_{\epsilon, 1}(s, x, y)dxdy \nonumber \\
&=& -r\int_{\mathbb{R}}\int_0^{+\infty}w^2(x)\left(I_{\epsilon, 1}^2(s, x, y)\right)_xdxdy \nonumber \\
&=& r\int_{\mathbb{R}}\int_0^{1}I_{\epsilon, 1}^2(s, x, y)dxdy \nonumber \\
&=& r\left\Vert \mathbb{I}_{\left[0, \, 1\right] \times \mathbb{R}}(\cdot)I_{\epsilon,1}(s,\cdot)\right\Vert _{\tilde{L}_{0}^2}^2
\end{eqnarray}
Next, by the definition of $u_{xx}$ in our SPDE we have
\begin{eqnarray}
& & \int_{\mathbb{R}^+}\int_{\mathbb{R}}u_{xx}(s,\,x,\,z)\phi_{\epsilon}(z,\,y)w^2(x)f(x)dzdx \nonumber \\
& & \qquad\qquad = \int_{\mathbb{R}^+}\int_{\mathbb{R}}u(s,\,x,\,z)\phi_{\epsilon}(z,\,y)\left(w^2(x)f(x)\right)_{xx}dzdx \nonumber \\
& & \qquad\qquad =-\int_{\mathbb{R}^+}\int_{\mathbb{R}}u_{x}(s,\,x,\,z)\phi_{\epsilon}(z,\,y)\left(w^2(x)f(x)\right)_{x}dzdx \nonumber \\
& & \qquad\qquad =-\int_{\mathbb{R}^+}\int_{\mathbb{R}}w^2(x)u_{x}(s,\,x,\,z)\phi_{\epsilon}(z,\,y)f_{x}(x)dzdx \nonumber \\
& & \qquad\qquad \qquad -\int_{\left[0, \, 1\right]}\int_{\mathbb{R}}u_{x}(s,\,x,\,z)\phi_{\epsilon}(z,\,y)f(x)dzdx, \label{eq:5.12}
\end{eqnarray}
which equals
\begin{eqnarray}
& & -\int_{\mathbb{R}^+}\int_{\mathbb{R}}w^2(x)u_{x}(s,\,x,\,z)\phi_{\epsilon}(z,\,y)f_{x}(x)dzdx \nonumber \\
& & \qquad +\int_{\left[0, \, 1\right]}\int_{\mathbb{R}}u(s,\,x,\,z)\phi_{\epsilon}(z,\,y)f_{x}(x)dzdx -\int_{\mathbb{R}}u(s,\,1,\,z)\phi_{\epsilon}(z,\,y)f(1)dz \nonumber
\end{eqnarray}
for any smooth function $f$ defined on $\left[0, \, +\infty\right)$. Since $u\in H_{0, w^2(x)}^{1}$ for any $y$ and since $f(1)$ can be controlled by the $H_{0, w^2(x)}^{1}$ norm of $f$ (by using Morrey's inequality near $1$), \eqref{eq:5.12} defines a linear functional on the space of smooth functions $f$ (defined on $\left[0, \, +\infty\right)$) which is bounded under the topology of $H_{0, w^2(x)}^{1}$. Then, since those functions form a dense subspace of $H_{0, w^2(x)}^{1}$, we deduce that \eqref{eq:5.12} holds for any $f \in H_{0, w^2(x)}^{1}$. Therefore, taking $f=I_{\epsilon,1}(s,\,\cdot,\,y)$ and integrating \eqref{eq:5.12} in $\left(y, t\right)$ over $\mathbb{R} \times \mathbb{R}^{+}$, we obtain
\begin{eqnarray}
&&\int_{0}^{t}\left\langle \frac{\partial^{2}}{\partial x^{2}}I_{\epsilon,h^{2}(z)}(s,\,\cdot),\,I_{\epsilon,1}(s,\,\cdot)\right\rangle _{\tilde{L}_{0,w}^2}ds \nonumber \\
&& \qquad =-\int_{0}^{t}\left\langle \frac{\partial}{\partial x}I_{\epsilon,h^{2}(z)}(s,\,\cdot),\,\frac{\partial}{\partial x}I_{\epsilon,1}(s,\,\cdot)\right\rangle _{\tilde{L}_{0,w}^2}ds. \nonumber \\
&& \qquad \qquad -\int_{0}^{t}\left\langle \frac{\partial}{\partial x}I_{\epsilon,h^{2}(z)}(s,\,\cdot),\,\mathbb{I}_{\left[0, \, 1\right] \times \mathbb{R}}(\cdot)I_{\epsilon,1}(s,\,\cdot)\right\rangle _{\tilde{L}_{0}^2}ds. \label{eq:5.13}
\end{eqnarray}
Finally, integrating by parts we find that
\begin{equation}
\int_{0}^{t}\left\langle \frac{\partial^{2}}{\partial y^{2}}I_{\epsilon,1}(s,\,\cdot),\,I_{\epsilon,1}(s,\,\cdot)\right\rangle _{\tilde{L}_{0,w}^2}ds = -\int_{0}^{t}\left\Vert \frac{\partial}{\partial y}I_{\epsilon,1}(s,\,\cdot)\right\Vert _{\tilde{L}_{0,w}^2}^{2}ds \label{eq:5.14}
\end{equation}
and
\begin{equation}
\int_{0}^{t}\left\langle \frac{\partial^{2}}{\partial x \partial y}I_{\epsilon,h\left(z\right)}(s,\cdot),\,I_{\epsilon,1}(s,\cdot)\right\rangle _{\tilde{L}_{0,w}^2}ds  = - \int_{0}^{t}\left\langle \frac{\partial}{\partial x}I_{\epsilon,h\left(z\right)}(s,\cdot),\,\frac{\partial}{\partial y}I_{\epsilon,1}(s,\cdot)\right\rangle _{\tilde{L}_{0,w}^2}ds
\label{eq:5.15}
\end{equation}
and also
\begin{equation}
\int_{0}^{t}\left\langle \frac{\partial}{\partial y}I_{\epsilon,g(z)}(s,\,\cdot),\,I_{\epsilon,1}(s,\,\cdot)
\right\rangle _{\tilde{L}_{0,w}^2}ds  = -\int_{0}^{t}\left\langle I_{\epsilon, g(z)}(s,\,\cdot),\,\frac{\partial}{\partial y}I_{\epsilon,1}(s,\,\cdot)
\right\rangle_{\tilde{L}_{0,w}^2}ds,
\label{eq:5.16}
\end{equation}
for $g(z) \in \{Q'(z), \, zQ'(z), \, q'(z)\}$. The steps we have followed are almost identical to those in the proof of \cite[Lemma~5.5]{HK17} and, as in there, we can justify that integration by parts in the $y$-direction does not leave any boundary term at infinity. Indeed, we can use Morrey's inequality in the same way and Lemma~\ref{lem:5.3} for large $\delta'$, in order to show that all the terms inside the inner products which do not involve $u_x$ are rapidly decreasing in $y$. Substituting \eqref{simpleIBP}, \eqref{eq:5.13} - \eqref{eq:5.15} and \eqref{eq:5.16} for $g(z) \in \{Q'(z), \, zQ'(z), \, q'(z)\}$ in \eqref{eq:5.11} and then taking expectations, we obtain \eqref{mainideq}, completing the proof of Lemma~\ref{mainid}.

\end{document}